\documentclass[10pt,a4,leqno]{amsart}
\usepackage{amsthm}
\usepackage{amsmath}
\usepackage{amssymb}
\usepackage{amsfonts}
\usepackage{mathrsfs}
\usepackage{graphicx}
\usepackage{bbm}
\usepackage[
colorlinks=true, citecolor=blue, linkcolor=blue, urlcolor=red]{hyperref}
\usepackage{color}
\usepackage[all]{xy}
\usepackage{comment}
\usepackage{here}
\usepackage{amscd}
\usepackage{caption}
\usepackage{listings}
\usepackage{setspace}
\usepackage{booktabs}
\allowdisplaybreaks

\newcommand{\mathsym}[1]{{}}
\newcommand{\unicode}[1]{{}}

\makeatletter

\@addtoreset{equation}{section}
\makeatother

\makeatletter
\@namedef{subjclassname@2020}{%
  \textup{2020} Mathematics Subject Classification}
\makeatother

\newtheoremstyle{my theoremstyle}
{1.0em}                    
    {1.0em}                    
    {\itshape}                   
    {}                           
    {\scshape}                   
    {.}                          
    {.5em}                       
    {}  

\newtheoremstyle{dfn}
{1.0em}                    
    {1.0em}                    
    {}                   
    {}                           
    {\scshape}                   
    {.}                          
    {.5em}                       
    {}  
    
\theoremstyle{my theoremstyle}
   \newtheorem{thm}{Theorem}[section]
   \newtheorem{lem}[thm]{Lemma}
   \newtheorem{prop}[thm]{Proposition}
   \newtheorem{cor}[thm]{Corollary}
   \newtheorem{conj}[thm]{Conjecture}
\theoremstyle{dfn}
   \newtheorem{dfn}[thm]{Definition}
   
\theoremstyle{remark}   
   
   \newtheorem{rmk}[thm]{{\scshape Remark}}

\newcommand{\Z}{\mathbb{Z}}
\newcommand{\Q}{\mathbb{Q}}

\newcommand{\C}{\mathbb{C}}
\newcommand{\F}{\overline{\mathbb{F}}_p}
\newcommand{\e}{\varepsilon}

\newcommand{\ub}{\underline{b}}
\newcommand{\ua}{\underline{a}}
\newcommand{\wF}{\widehat{\mathscr{F}}^{(\sigma)}}

\newcommand{\we}{\widehat{\varepsilon}}
\newcommand{\wE}{\widehat{E}}
\newcommand{\of}{\overline{f}}
\newcommand{\al}{\alpha}

\newcommand{\od}{\overline{d}}

\numberwithin{equation}{section}

\date{\today}

\begin{document}
\title[On the syntomic regulator and $p$-adic hypergeometric functions]{On the syntomic regulator of the Hesse cubic curves and $p$-adic hypergeometric functions}
\author{Yusuke Nemoto}
\date{\today}
\address{Graduate School of Science, Chiba University, 
Yayoicho 1-33, Inage, Chiba, 263-8522, Japan.}
\email{y-nemoto@waseda.jp}

\maketitle

\begin{abstract}
We introduce a new type of $p$-adic hypergeometric function, which satisfies congruence relations similar to Dwork's $p$-adic hypergeometric function. 
Also, we prove that the syntomic regulators of the Hesse cubic curves are expressed in terms of the special values of our new $p$-adic hypergeometric functions.  
We also show that there is a transformation formula between our new $p$-adic hypergeometric function and a $p$-adic hypergeometric function of logarithmic type defined by Asakura. 
\end{abstract}

\section{Introduction} 
Let $p$ be a prime and $s \ge 1$ be an integer. 
For an $s$-tuple $\ua=(a_1, \ldots, a_{s}) \in \Z_p^s$ and an $(s-1)$-tuple $\ub=(b_1, \ldots, b_{s-1}) \in (\Z_p \setminus \Z_{\le 0})^{s-1}$, 
we define the {\it hypergeometric series} by 
\begin{align*}
F_{\ua; \ub}(t)={_{s}F_{s-1}}\left( 
\begin{matrix}
a_1, \cdots, a_s \\
b_1, \ldots b_{s-1}
\end{matrix}
; t
\right)=\sum_{n=0}^{\infty} \dfrac{(a_1)_n \cdots (a_s)_n}{(1)_n (b_1)_n \cdots (b_{s-1})_n} t^n \in \Q_p[[t]],  
\end{align*}
where $(a)_n$ denotes the Pochhammer symbol defined by $(a)_n=a (a+1) \cdots (a+n-1)$. 
Dwork \cite{Dw, Dw2} shows that if $\ub=\underline{1}:=(1, \ldots, 1)$ or $\ua$ and $\ub$ satisfy the conditions (i) and (ii) in Theorem \ref{theorem:1}, then $F_{\ua; \ub}(t) \in \Z_p[[t]]$. 
In this paper, we always assume these conditions.

For $a \in \Z_p$, let $a'$ be the {\it Dwork prime} defined by $(a+l)/p$, where $l \in \{0, \ldots, p-1\}$ is the unique integer such that $a + l \equiv 0 \pmod{p}$. 
Let $a^{(i)}$ be the {\it $i$th Dwork prime} defined by $a^{(i)}=(a^{(i-1)})'$ with $a^{(0)}=a$. 
Put $\ua'=(a'_1, \ldots, a'_s)$ and $\ub'=(b'_1, \ldots, b'_{s-1})$.

Dwork defines the $p$-adic hypergeometric function by $\mathscr{F}^{\rm Dw}_{\ua; \ub}(t)=F_{\ua; \ub}(t)/F_{\ua'; \ub'}(t^p)$ and proves the congruence relation 
\begin{align} \label{congruence}
\mathscr{F}^{\rm Dw}_{\ua; \ub}(t) \equiv \dfrac{[F_{\ua; \ub}(t)]_{<p^n}}{[F_{\ua'; \ub'}(t^p)]_{<p^n}} \pmod{p^n\Z_p[[t]]}
\end{align}
for any $n \ge 1$ (see Theorem \ref{theorem:1}). Here, for a power series $f(t)=\sum a_i t^i$, we write $[f(t)]_{<m}=\sum_{i < m} a_it^i$ for its truncation.  
As a consequence of this congruence, $\mathscr{F}^{\rm Dw}_{\ua; \ub}(t)$ defines a $p$-adic analytic function on a suitable domain, where its special values can be defined. 
Such special values occur in formulas for  the unit roots of hypergeometric motives (cf. \cite{Dw, As, As3}). 
Other $p$-adic hypergeometric functions, such as $\mathscr{F}^{(\sigma)}_{\ua; \ub}(t)$ and $\wF_{a, \ldots, a;\underline{1}}(t)$, are also defined by Asakura \cite{As}, the author \cite{N2} and Wang \cite{Wang}. 
These functions satisfy congruence relations analogous to \eqref{congruence}. 
From a geometric perspective, the syntomic regulators of hypergeometric motives are expressed in terms of these functions (\cite{Wang, As, As3}).

In this paper, we introduce a new type of $p$-adic hypergeometric function for $s=2$, which is a modification of Wang's $p$-adic hypergeometric function $\wF_{a, a; 1}(t)$. 
Let $p \ge 3$ be a prime. 
Let $W=W(\F)$ be the Witt ring of $\F$ and $\sigma$ be the $p$th Frobenius on $W[[t]]$ defined by $\sigma(t)= ct^{p}$ with $c \in 1+pW$. 
Suppose that $a \in \Z_p$ satisfies two conditions in Section \ref{defFhat}. Under these conditions, we have $F_{a, a; 2a}(t) \in \Z_p[[t]]$. 
Let $l \in \{0, \dots p-1 \}$ be the unique integer such that $a+l \equiv 0 \pmod{p}$. 
Define a power series by 
\begin{align*}
\widehat{G}^{(\sigma)}_{a, a;2a}(t) &=t^{-a} \int_0^t(t^a F_{a,a;2a} (t) - (-1)^lB_p(a,a)^{-1} [t^{a'} F_{a',a'; (2a)'}(t)]^{\sigma}) \dfrac{dt}t, 
\end{align*}
where $B_p(s, s)=\Gamma_p(s)^2/\Gamma_p(2s)$ is the {\it $p$-adic Beta function} (see Section \ref{definition}).  
Here, $\int_0^t(-)\frac{dt}t$ denotes the operator determined by  
$$\int_0^t t^{\al} \dfrac{dt}{t}= \frac{t^{\al}}{\al}, \quad \al \neq 0. $$ 
We define the {\it $p$-adic hypergeometric function} by 
$$\widehat{\mathscr{F}}_{a,a;2a}^{(\sigma)} (t)= \dfrac{\widehat{G}^{(\sigma)}_{a, a;2a}(t) }{F_{a,a;2a}(t)}. $$
If $a=1/2$, then the above function agrees with Wang's $p$-adic hypergeometric function (see Remark \ref{relation_Wang}).

We prove that $\widehat{\mathscr{F}}_{a,a;2a}^{(\sigma)} (t) \in W[[t]]$ (see Lemma \ref{integral}) and the following congruence relation.

\begin{thm} \label{main:1}
For any $n \in \Z_{\ge 1}$, we have  
$$\widehat{\mathscr{F}}_{a,a;2a}^{(\sigma)} (t) \equiv \dfrac{[\widehat{G}^{(\sigma)}_{a,a;2a}(t)]_{<p^n}}{[F_{a,a;2a}(t)]_{<p^n}} \pmod{p^nW[[t]]}. $$
\end{thm}

A sketch of the proof of Theorem \ref{main:1} is as follows. 
First, we reduce to the case $c=1$ (see Lemma \ref{reduced}). 
Write $F_{a, a; 2a}(t)=\sum_{k=0}^{\infty} A_k t^k$ and $\widehat{G}_{a, a; 2a}(t)=\sum_{k=0}^{\infty} B_k t^k$. 
It suffices to show that for any $m, n \in \Z_{\ge 0}$, 
$$\sum_{i+j=m}A_{i+p^n} B_j - A_i B_{j+p^n} \equiv 0 \pmod{p^n}. $$
To prove this, we employ the method used in \cite{N2}, which is a slight modification of that in \cite{Wang}. 
The key point of this method is to show that if $m_1 \equiv m_2 \pmod{p^n}$, then 
$$\dfrac{B_{m_1}}{A_{m_1}} \equiv \dfrac{B_{m_2}}{A_{m_2}} \pmod{p^n}$$
(see Lemma \ref{Lipschitz}).

The second main result is to give a geometric aspect of our function $\wF_{a, a; 2a}(t)$, which concerns the {\it syntomic regulator map}. 
Let $p$ be a prime such that $p \equiv 1 \pmod{3}$. 
Let $K=\operatorname{Frac}(W)$ be the fractional field of $W$. 
Let $\al \in W^{\times}$ satisfy $\al^3 \not \equiv 1 \pmod{p}$, and $X_{\alpha}$ the Hesse cubic curve over $W$ defined by 
$$x^3_0+y^3_0+z^3_0-3\al x_0y_0z_0=0. $$
Put $X_{\al, K}:=X_{\alpha} \times_W K$. 
In this paper, we study the syntomic regulator map (\cite{Besser, Nekovar_Niziol}) 
$$r_{\rm syn} \colon K_2(X_{\alpha, K})^{(2)} \otimes \Q \to H^2_{\rm syn}(X_{\alpha, K}, \Q_p(2)) \simeq H^1_{\rm dR}(X_{\al, K}/K). $$

Our second main result is to express the syntomic regulator of a certain element of the $K_2$-group of the Hesse cubic curve in terms of the special value of $\wF_{a, a; 2a}(t)$.  
Put $x=x_0/z_0$ and $y=y_0/z_0$.

\begin{thm}[Corollary \ref{main:cor}] \label{Intro_main2}
Suppose that $[F_{\frac13, \frac13; \frac23}(t)]_{<p}|_{t=\alpha^3} \not \equiv 0 \pmod{p}$. 
Let 
$$\xi_{\al}=6\{x, y\} \in K_2(X_{\alpha, K})^{(2)} \otimes \Q$$
be a $K_2$-symbol.  
Let 
$$Q \colon H^1_{\rm dR}(X_{\alpha, K}/K) \otimes H^1_{\rm dR}(X_{\alpha,K}/K) \to H^2_{\rm dR}(X_{\alpha, K}/K) \simeq K$$
denote the cup-product pairing. 
Let $\omega_t=dy/(x^2-ty)$ be a regular $1$-form (see Definition \ref{diff_form}), and let $e_t^{\rm unit}$ denote the unit root vectors explicitly given in Theorem \ref{unit_root_formula}.  
Then, we have 
$$Q(r_{\rm syn}(\xi_{\al}), e_{\alpha}^{\rm unit})=-  2\al \widehat{\mathscr{F}}_{\frac13, \frac13; \frac23}^{(\tau_{\alpha})}(\alpha^3) Q (\omega_{\alpha}, e_{\alpha}^{\rm unit}). $$
\end{thm}

Over the complex numbers, the author \cite{N} proves that the regulator of the Hesse cubic curve $r_{\mathscr{D}, \Q}(\xi_{\al})$ is expressed in terms of generalized hypergeometric functions. 
More precisely, we have 
\begin{align} \label{complex_regulator}
\begin{split}
&\dfrac1{2 \pi i} r_{\mathscr{D}, \Q}(\xi_{\al})(\gamma) \\
& \qquad = \dfrac{3\sqrt{3}}{\pi} \left( B\left(\frac13, \frac13 \right)
\al {_{3}F_{2}}\left( 
\begin{matrix}
\frac13, \frac13, \frac13 \\
\frac43, \frac23
\end{matrix}
; \al^3
\right) 
+ \frac12 \al^2B\left(\frac23, \frac23 \right){_{3}F_{2}}\left( 
\begin{matrix}
\frac23, \frac23, \frac23 \\
\frac53, \frac43
\end{matrix}
; \al^3
\right) 
\right), 
\end{split}
\end{align}
where $B(x, y)$ denotes the Beta function\footnote{This formula is also proved using the Mahler measure (see \cite{Deninger, Rod}). }.
Thus, Theorem \ref{Intro_main2} is a $p$-adic counterpart of the formula \eqref{complex_regulator}.

Our third main result concerns a transformation formula between $p$-adic hypergeometric functions. 
In previous work, Wang \cite{Wang} conjectures that there is a transformation formula between Dwork's $p$-adic hypergeometric functions $\mathscr{F}^{\rm Dw}_{a, \ldots, a; \underline{1}}(t)$ and $\mathscr{F}^{\rm Dw}_{a, \ldots, a; \underline{1}}(t^{-1})$. Under some assumptions, Wang \cite{Wang, Wang2} proves 
the conjecture for $s=1, 2$, and the author \cite{N3} proves it for general $s \ge 2$. 
Wang also conjectures that there is a transformation formula relating $\mathscr{F}_{a, \ldots,a; \underline{1}}^{(\sigma)}(t)$, which is the $p$-adic hypergeometric function of logarithmic type defined by Asakura, to the \(p\)-adic hypergeometric function $\wF_{a, \ldots, a; \underline{1}}(t^{-1})$ introduced by Wang.  
Wang proves that the conjecture is true for $s=2$ in \cite{Wang} and for general $s \ge 1$ in \cite{Wang2} under some assumptions.

In this paper, we consider a transformation formula between $\mathscr{F}_{a, 1-a;1}^{(\sigma)}(t)$ and  $\wF_{a, a; 2a}(t^{-1})$. 
For a polynomial $h(t) \in W[t]$, we put 
$$W\langle t, t^{-1},  h(t)^{-1}\rangle := \varprojlim_{n \ge 1} (W/p^nW[t, t^{-1}, h(t)^{-1}]). $$
Let $r$ be a positive integer such that $a^{(r)}=a$. 
Define polynomials by 
$$h_a(t)=\prod_{i=0}^{r-1}[F^{(i)}_{a,1-a;1}(t)]_{<p}, \quad \widehat{h}_a(t)=\prod_{i=0}^{r-1}[F^{(i)}_{a,a;2a}(t)]_{<p}. $$
Then, there is an isomorphism (see Proposition \ref{involution}) 
$$\omega \colon W \langle t, t^{-1}, \widehat{h}_a(t)^{-1} \rangle \to W \langle t, t^{-1}, h_a(t)^{-1} \rangle; \quad \omega(f(t))=f(t^{-1}). $$
We propose the following conjecture.

\begin{conj} \label{Intro_conj}
Let $\sigma(t)=ct^p$ and $\widehat{\sigma}(t)=c^{-1}t^p$ with $c \in 1+ pW$.   
Then we have  
\begin{align*}
\mathscr{F}^{(\sigma)}_{a, 1-a; 1}(t) = - \widehat{\mathscr{F}}^{(\widehat{\sigma})}_{a, a; 2a}(t^{-1})
\end{align*}
in the ring $W\langle t, t^{-1}, h_a(t)^{-1} \rangle$, where $\widehat{\mathscr{F}}^{(\widehat{\sigma})}_{a, a; 2a}(t^{-1}):=\omega(\widehat{\mathscr{F}}^{(\widehat{\sigma})}_{a, a; 2a}(t))$.  
\end{conj}

We prove that this conjecture holds modulo $p$ (see Theorem \ref{conj_mod_p}). 
We also prove the conjecture for $a=1/3$ by comparing the syntomic regulators of the Hesse cubic curves (see Theorem \ref{main:3}).

Over the complex numbers, Zudilin's formula \cite[Lemma 3.5]{As4} gives the following transformation formula: 
\begin{align} \label{trans_overC}
\begin{split}
&\pi i + 2 \psi(1) -\psi(a) -\psi (1-a) -\log t -a(1-a)t {_{4}F_{3}}\left( 
\begin{matrix}
a+1, 2-a, 1, 1 \\
2, 2, 2
\end{matrix}
; t
\right)
\\
& = \dfrac{\tan(\pi a)}{ 2 \pi a}  B(a, a) (-t^{-1})^a 
{_{3}F_{2}}\left( 
\begin{matrix}
a, a, a \\
2a, a+1
\end{matrix}
; t^{-1}
\right) \\
& \quad +
\dfrac{\tan(\pi (1-a))}{ 2 \pi (1-a)}  B(1-a, 1-a) (-t^{-1})^{1-a} 
{_{3}F_{2}}\left( 
\begin{matrix}
1-a, 1-a, 1-a \\
2-2a, 2-a
\end{matrix}
; t^{-1}
\right), 
\end{split}
\end{align}
where $\psi(x)=\Gamma'(x)/\Gamma(x)$ denotes the digamma function.  
The $p$-adic hypergeometric function of logarithmic type $\mathscr{F}^{(\sigma)}_{a, 1-a; 1}(t)$ is a $p$-adic analogue of the expression on the left-hand side of \eqref{trans_overC} (see \eqref{Asakura_function}). 
Hence, Conjecture \ref{Intro_conj} is a $p$-adic counterpart of \eqref{trans_overC}.

This paper is constructed as follows. 
In Section \ref{Dwork_pHG}, we recall Dwork's $p$-adic hypergeometric function and its congruence relations. 
In Section \ref{several_pHG}, we recall several $p$-adic hypergeometric functions, which are defined by Asakura and Wang. 
In Section \ref{definition}, we give a definition of a new type of $p$-adic hypergeometric function $\wF_{a, a;2a}(t)$.  
In Section \ref{main_proof}, we give a proof of Theorem \ref{main:1}, i.e., the congruence relations for $\wF_{a, a;2a}(t)$. 
In Section \ref{syntomic_reg_Hesse}, we show that our new $p$-adic hypergeometric function appears in formulas for the syntomic regulators of the Hesse cubic curves. 
In Section \ref{main:transform}, we introduce a conjectural transformation formula relating $\mathscr{F}^{(\sigma)}_{a,1-a;1}(t)$ to $\widehat{\mathscr{F}}^{(\widehat{\sigma})}_{a,a;2a}(t^{-1})$, and prove it for $a=1/3$.

\section{Dwork's $p$-adic hypergeometric function} \label{Dwork_pHG}
In this section, we recall $p$-adic hypergeometric functions defined by Dwork in \cite{Dw, Dw2}.

For $\alpha \in \Z_p$, let $\alpha'$ be the Dwork prime of $\al$ defined by $\alpha'=(\alpha+l)/p$, where $l \in \{0, 1, \ldots, p-1\}$ is the unique integer such that $\alpha+l \equiv 0 \pmod{p}$. To put it another way, if $p$-adic integer $\alpha$ has the $p$-adic expansion 
$$\alpha=-[\alpha]_0-[\alpha]_1p-[\alpha]_2p^2-\cdots -[\alpha]_np^n-\cdots$$
where $[\alpha]_i \in \{0,1, \cdots, p-1\}$, then $\alpha^{\prime}$ is defined by
$$\alpha^{\prime}=-[\alpha]_1-[\alpha]_2p-[\alpha]_3p^2-\cdots -[\alpha]_{n+1}p^n-\cdots.$$
For $i \ge 0$, let $a^{(i)}$ denote the $i$th Dwork prime defined by $a^{(i)}=(a^{(i-1)})^{\prime}$ with $a^{(0)}=a$. 
Put 
$$F_{\ua; \ub}(t)={_{s}F_{s-1}}\left( 
\begin{matrix}
a_1, \cdots, a_s \\
b_1, \ldots b_{s-1}
\end{matrix}
; t
\right), \quad F^{(i)}_{\ua; \ub}(t)={_{s}F_{s-1}}\left( 
\begin{matrix}
a_1^{(i)}, \cdots, a_s^{(i)} \\
b_1^{(i)}, \ldots b_{s-1}^{(i)}
\end{matrix}
; t
\right). $$
Then Dwork \cite{Dw2} proves the following theorem.

\begin{thm}[{\cite[Theorem 1.1]{Dw2}}, cf. {\cite[Theorem 2.3]{Young}}]\label{theorem:1}
Let $\ua=(a_1, \ldots, a_s) \in (\mathbb{Q} \cap \mathbb{Z}_p)^s$ and $\ub=(b_1, \ldots, b_{s-1}) \in (\mathbb{Q} \cap \mathbb{Z}_p\backslash \mathbb{Z}_{\leq 0})^{s-1}$. 
Suppose that, for some $q$ with $0\le q\le s-1$,
\[
b_j\ne1\quad(1\le j\le q),
\qquad
b_j=1\quad(q<j\le s-1),
\]
and that the following conditions hold: 
\begin{enumerate}
\item $|b^{(i)}_j|_p=1$ for all $i \geq 0$, $j=1,\ldots, q$,  
\item For each fixed $n$, after rearranging the indices so that $[a_{1}]_n \leq \cdots \leq [a_{s}]_n$ and $[b_1]_n \leq \cdots \leq [b_{s-1}]_n$, 
we have 
$$  [a_{j+1}]_n < [b_j]_n \quad (j=1, \ldots, q).$$
\end{enumerate}
\noindent Then for any $i \geq 0$, we have $\displaystyle F^{(i)}_{\ua; \ub}(t) \in \mathbb{Z}_p[[t]]$ and the following congruence holds: 
$$F_{\ua; \ub}(t) [F^{(1)}_{\ua; \ub}(t^p)]_{<p^n} \equiv F^{(1)}_{\ua; \ub}(t^p) [F_{\ua; \ub}(t)]_{<p^n}    \pmod{p^{l} \mathbb{Z}_p[[t]]} \quad  (n \geq l \ge 1). $$
\end{thm}

Dwork defines the $p$-adic hypergeometric function as the ratio of hypergeometric series  
\begin{align} \label{Dwork_congruence}
\mathscr{F}_{\ua; \ub}^{\rm Dw}(t)=\dfrac{F_{\ua; \ub}(t)}{F^{(1)}_{\ua; \ub}(t^p)}. 
\end{align}
By Theorem \ref{theorem:1}, we have the congruence relations
\begin{align*}
\mathscr{F}_{\ua; \ub}^{\rm Dw}(t) \equiv \dfrac{[F_{\ua; \ub}(t)]_{<p^n}}{[F^{(1)}_{\ua; \ub}(t^p)]_{<p^n}} \pmod{p^n\mathbb{Z}_p[[t]]}. 
\end{align*}
Therefore, this function is $p$-adically analytic in the sense of Krasner, which means that this function defines an element of the Tate algebra.

\section{Several $p$-adic hypergeometric functions} \label{several_pHG}
\subsection{$p$-adic hypergeometric function of logarithmic type} \label{p_adic_logHGF}
In this subsection, we briefly recall a $p$-adic hypergeometric function of logarithmic type, which is defined by Asakura \cite{As} for $\ub=\underline{1}$, and by the author \cite{N2} for a general $\ub$. 
These functions satisfy congruence relations similar to Dwork's. 
In this paper, we only recall the case $\ub=\underline{1}$. 
\begin{dfn} \label{definition:3}
For $z \in \mathbb{Z}_p$, we define 
$$\widetilde{\psi}_p(z) = \displaystyle \lim_{n \in \mathbb{Z}_{>0}, n \rightarrow z} \sum _{1 \leq k <n, p \nmid  k}\dfrac{1}{k}.$$
Define the \textit{$p$-adic Euler constant} by 
$$\gamma_p=-\displaystyle \lim_{s \to \infty}\dfrac{1}{p^s} \sum_{0 \leq j<p^s, p \nmid j }\log(j),$$
where $\log \colon \C^{\times}_p \to \C_p$ is the \textit{Iwasawa logarithmic function}, which is characterized by a continuous homomorphism satisfying $\log(p)=0$ and 
$$\log(z)=-\displaystyle \sum_{n=1}^{\infty} \dfrac{(1-z)^n}{n}, \quad |z-1|_p<1.$$ 
We define the \textit{$p$-adic digamma function} by
$$\psi_p(z)=-\gamma_p +\widetilde{\psi}_p(z), $$
which is a $p$-adic continuous function on $\mathbb{Z}_p$ (see \cite[Lemma 2.5 and (2.13)]{As}).
\end{dfn}

Let $\sigma$ be a $p$th Frobenius on $W[[t]]$ defined by $\sigma(t)=ct^p$ ($c \in 1+pW$), i.e.
$$\left(\sum_i a_i t^i \right)^{\sigma}=\sum_i a_i^Fc^it^{ip}, $$
where $F$ is the Frobenius on $W$.  
For $\ua \in \Z_p^s$, put 
$$G_{\ua; \underline{1}}(t)= \sum_{i=1}^{s}\psi_p(a_i)+ s\gamma_p-p^{-1} \log(c) +\int_0^t (F_{\ua;\underline{1}}(t)-F^{(1)}_{\ua; \underline{1}}(t^{\sigma}))\dfrac{dt}{t}, $$
where $\int_0^t(-)\frac{dt}{t}$ is an operator which sends a power series $\sum_{n \ge 1} a_n t^n$ to $\sum_{n \ge 1} \frac{a_n}{n} t^n$. 
Then Asakura \cite{As} defines the \textit{$p$-adic hypergeometric function of logarithmic type} by 
\begin{align} \label{Asakura_function}
\mathscr{F}_{\ua; \underline{1}}^{(\sigma)}(t)=\dfrac{G_{\ua; \underline{1}}(t)}{F_{\ua; \underline{1}}(t)} \in W[[t]]
\end{align}
and proves the following congruence relations similar to Dwork's \eqref{Dwork_congruence}.     
\begin{thm}[{\cite[Theorem 3.3]{As}}]
Suppose that $a_i \not \in \mathbb{Z}_{\leq 0}$ for all $i$. If $p$ is odd, then for all $n \geq 1$, we have
\begin{align*}
\mathscr{F}_{\ua; \underline{1}}^{(\sigma)}(t) \equiv \dfrac{[G_{\ua; \underline{1}}(t)]_{<p^n}}{[F_{\ua; \underline{1}}(t)]_{<p^n}} \pmod{p^n{W}[[t]]}.   
\end{align*}
If $p=2$, then the congruence above holds modulo $p^{n-1}W[[t]]$.
\end{thm}

Suppose that there exists an integer $r >0$ such that $\ua^{(r)}=\ua$. 
Then the theorem implies that $\mathscr{F}^{(\sigma)}_{\ua; \underline{1}}(t)$ is a $p$-adic analytic function in the sense of Krasner, i.e. we have 
$$\mathscr{F}^{(\sigma)}_{\ua; \underline{1}}(t) \in W\langle t, h_{\ua; \underline{1}}(t)^{-1} \rangle:= \varprojlim_{n \geq 1}(W/p^nW[t, h_{\ua; \underline{1}}(t)^{-1}]), \quad h_{\ua; \underline{1}}(t)=\prod_{i=0}^{r-1}[F^{(i)}_{\ua; \underline{1}}(t)]_{<p}$$
(see \cite[Corollary 3.4]{As}).

\subsection{$p$-adic hypergeometric function $\wF_{a, \ldots, a; \underline{1}}(t)$} \label{Wang_pHGF} 
In this subsection, we recall a new kind of $p$-adic hypergeometric function, which is defined by Wang \cite{Wang}. 
Put 
$$q:=
\begin{cases}
4 & p=2,  \\
p & p \ge 3. 
\end{cases}
$$
Let $l' \in \{0, \ldots, q-1\}$ be the unique integer such that $a + l' \equiv 0 \pmod{q}$. 
Put 
$$e:= l' - \lfloor \dfrac{l'}p \rfloor. $$ 
Let $a_1 = \cdots = a_s=a$. Then Wang defines a new type of $p$-adic hypergeometric function $\wF_{a, \ldots, a; \underline{1}}(t)$ by 
\begin{align*}
\wF_{a, \ldots, a; \underline{1}}(t) = \dfrac{t^{-a}}{ F_{a, \ldots, a; \underline{1}}(t)} \int_0^t (t^a F_{a, \ldots, a; \underline{1}}(t) -(-1)^{se} [t^{a'}F_{a', \ldots, a'; \underline{1}}(t)]^{\sigma}) \dfrac{dt}{t} \in W[[t]].  
\end{align*}
Write $\wF_{a, \ldots, a; \underline{1}}(t):=\widehat{G}^{(\sigma)}_{a, \ldots, a; \underline{1}}/F_{a, \ldots, a; \underline{1}}(t)$. 
Then Wang proves the following congruence relations.

\begin{thm}[{\cite[Theorem 3.1]{Wang}}]
Let $a \in \Z_p \setminus \Z_{\le 0}$ and suppose $c \in 1 + qW$. Then, 
$$\wF_{a, \ldots, a; \underline{1}}(t) \equiv \dfrac{[\widehat{G}^{(\sigma)}_{a, \ldots, a; \underline{1}}(t)]_{<p^n}}{[F_{a, \ldots, a; \underline{1}}(t)]_{<p^n}} \pmod{p^nW[[t]]}$$
for all $n \in \Z_{\ge 0}$. 
\end{thm}

Suppose that there exists an integer $r > 0$ such that $a^{(r)}=a$. 
Then the theorem implies that $\wF_{a, \dots, a; \underline{1}}(t)$ is a $p$-adic analytic function in the sense of Krasner, i.e. we have 
\begin{align*}
\mathscr{\widehat{F}}^{(\sigma)}_{a, \ldots, a; \underline{1}}(t) \in W\langle t, h_{a, \ldots, a; \underline{1}}(t)^{-1} \rangle:= \varprojlim_{n \geq 1}(W/p^nW[t, h_{a, \ldots, a; \underline{1}}(t)^{-1}]),
\end{align*}
where  
$$h_{a, \ldots, a; \underline{1}}(t)=\prod_{i=0}^{r-1}[F^{(i)}_{a, \ldots, a; \underline{1}}(t)]_{<p}$$ 
(see \cite[Corollary 3.2]{Wang}).

\section{Definition of $\wF_{a, a;2a}(t)$} \label{definition}
\subsection{$p$-adic Gamma function and $p$-adic Beta function}
Let $p$ be an odd prime. 
For $z \in \Z_p$, {\it Morita's $p$-adic Gamma function} $\Gamma_p \colon \Z_p \to \Z_p^{\times}$ (\cite{Morita}, cf. \cite[Section IV, 2]{Koblitz}) is defined by 
$$\Gamma_p(z)= \lim_{n \in \Z_{\ge0}, n \to z}(-1)^n \prod_{\substack{0 < i < n \\ p \nmid i}} i. $$
If $z_1 \equiv z_2 \pmod{p^s}$, then we have 
\begin{align} \label{p_adic_gamma}
\Gamma_p(z_1) \equiv \Gamma_p(z_2) \pmod{p^s}, 
\end{align}
hence this function is a $p$-adic continuous function on $\Z_p$. 
The $p$-adic Gamma function satisfies the reflection formula 
\begin{align} \label{reflection}
& \Gamma_p(x) \Gamma_p(1-x)=-(-1)^{l} \qquad (x \in \Z_p),  
\end{align} 
where $l \in \{0, \ldots, p-1\}$ is the unique integer such that $x+l \equiv 0 \pmod{p}$. 
For $\al \in \Z_p$ and $n \in \Z_{\ge 1}$, we define 
 \begin{align} \label{p_pochhammer}
\{\al\}_n = \prod_{\substack{1 \le i \le n \\ p \nmid (\al +  i - 1)}} (\al +  i -1)
 \end{align}
 and $\{\al\}_0=1$. 
Then, one has 
$$\{\al\}_n= (-1)^n \dfrac{\Gamma_p(\al+n)}{\Gamma_p(\al)}. $$
For $x, y \in \Z_p$, we define the {\it $p$-adic Beta function} $B_p(x, y)$ by 
\begin{align*}
B_p(x, y)=\dfrac{\Gamma_p(x) \Gamma_p(y)}{\Gamma_p(x+y)}. 
\end{align*}
Since $\Gamma_p(z)\in\Z_p^\times$ for any $z\in\Z_p$, we have $B_p(x,y)\in\Z_p^\times$.

\subsection{Definition of $\wF_{a,a;2a}(t)$} \label{defFhat}
From now on, for a $p$-adic integer 
$$a=-[a]_0-[a]_1p-[a]_2p^2-\cdots -[a]_np^n-\cdots, $$
we suppose the following two conditions: 
\begin{enumerate}
\item $a \in N^{-1} \Z$ and $p \nmid N$ for some integer $N \ge 2$, 
\item for any $i \in \Z_{\ge 0}$, $[a]_i < [2a]_i$, i.e., $0 < [a]_i \leq (p-1)/{2}$. 
\end{enumerate}
Under these conditions, we have $F_{a,a;2a}(t) \in \Z_p[[t]]$ (see Theorem \ref{theorem:1}).  
Let $l \in \{0, \dots p-1 \}$ be the unique integer such that $a+l \equiv 0 \pmod{p}$.

\begin{dfn}
Put 
\begin{align*}
\widehat{G}^{(\sigma)}_{a, a;2a}(t) & =t^{-a} \int_0^t(t^a F_{a,a;2a} (t) - (-1)^lB_p(a,a)^{-1} [t^{a'} F_{a',a'; (2a)'}(t)]^{\sigma}) \dfrac{dt}t. 
\end{align*}
Define a new type of $p$-adic hypergeometric function by 
$$\widehat{\mathscr{F}}_{a,a;2a}^{(\sigma)} (t) = \dfrac{\widehat{G}^{(\sigma)}_{a, a;2a}(t) }{F_{a,a;2a}(t)}. $$
\end{dfn}
If we write $F_{a,a;2a}(t) =\sum_{k=0}^{\infty} A_k t^k$, $F_{a',a ';(2a)'}(t) =\sum_{k=0}^{\infty} A^{(1)}_k t^k$ and $\widehat{G}^{(\sigma)}_{a, a;2a}(t)=\sum_{k=0}^{\infty} B_k t^k$, then we have 
$$B_k=\dfrac1{k+a} \left(A_k - (-1)^lB_p(a,a)^{-1} A^{(1)}_{\frac{k-l}p} c^{\frac{k+a}p} \right), $$
where $A^{(1)}_{\frac{m}p}=0$ if $m \not \equiv 0 \pmod{p}$ or $m <0$.

\begin{rmk} \label{relation_Wang}
When $a = 1/2$, the above definition agrees with Wang's $p$-adic hypergeometric function defined in Section \ref{Wang_pHGF}. 
Indeed, since $\Gamma_p(1)=-1$, $\Gamma_p(\frac12)^2=-\left(\frac{-1}p \right)=-(-1)^{\frac{p-1}2}$ and $l=(p-1)/2$, we have 
\begin{align*} 
(-1)^lB_p\left(\frac12, \frac12 \right)^{-1}=(-1)^l \cdot  \dfrac{\Gamma_p(1)}{\Gamma_p(\frac12)^2}=(-1)^{l-\frac{p-1}2}=1. 
\end{align*}
On the other hand, for Wang's function, we have $s=2$ and $e=l$, so that
\[
(-1)^{se}=(-1)^{2l}=1.
\]
Thus, the two definitions coincide.
\end{rmk}

\begin{lem} \label{integral}
For any $k \in \Z_{\ge 0}$, we have $B_k \in W$. Consequently, 
$$\widehat{\mathscr{F}}_{a,a;2a}^{(\sigma)} (t)  \in W[[t]]. $$
\end{lem}

\begin{proof}
If $k \not \equiv l \pmod{p}$, then $k+a \not \equiv 0 \pmod{p}$. 
Hence we have 
$$B_k = \dfrac{A_k}{k+a} \in W. $$
For $k \equiv l \pmod{p}$, we write $k+a=d p^m$ with $p \nmid d$. 
Then we have 
$$c^{\frac{k+a}p}= c^{p^{m-1} d} \equiv 1 \pmod{p^m}$$
since $c \in 1+pW$. 
Hence, it suffices to show 
$$A_k \equiv (-1)^lB_p(a, a)^{-1} A^{(1)}_{\frac{k-l}p} \pmod{p^m}. $$
By \cite[Lemma 3.6]{As}, we have 
$$A_k (A^{(1)}_{\frac{k-l}p})^{-1} =\dfrac{\{a\}_k \{a\}_k}{\{2a\}_k \{1\}_k}  =  \left(\dfrac{\{a\}_k}{\{1\}_k} \right)^2 \dfrac{\{1\}_k}{\{2a\}_k}. $$
As in the proof of \cite[Lemma 2.2]{Wang}, we have 
$$\dfrac{\{a\}_k}{\{1\}_k} \equiv (-1)^{k - \lfloor k/p \rfloor} \pmod{p^m}. $$
On the other hand, by \eqref{p_pochhammer} and \eqref{reflection}, we have 
\begin{align*}
\frac{\{1\}_k}{\{2a\}_k} 
&= \dfrac{\Gamma_p(1+k) \Gamma_p(2a)}{\Gamma_p(1) \Gamma_p(2a+k)}
\equiv \dfrac{\Gamma_p(1-a)\Gamma_p(2a)}{\Gamma_p(1) \Gamma_p(a)} = (-1)^l B_p(a, a)^{-1} \pmod{p^m}, 
\end{align*}
which finishes the proof. 
\end{proof}

One of the main theorems in this paper is that $\wF_{a,a;2a}(t)$ satisfies congruence relations similar to Wang's (see Theorem \ref{main:1}). 
Let $r$ be a positive integer such that $a^{(r)}=a$. 
By applying this theorem and arguing as in the proof of \cite[Corollary 3.4]{As} or \cite[Corollary 3.2]{Wang}, we similarly obtain that $\wF_{a,a;2a}(t)$ is a $p$-adic analytic function in the sense of Krasner, i.e., we have
\begin{align} \label{convergent}
\wF_{a,a;2a}(t) \in W\langle t, \widehat{h}_a(t)^{-1} \rangle, \quad \widehat{h}_a(t)=\prod_{i=0}^{r-1}[F^{(i)}_{a,a;2a}(t)]_{<p}. 
\end{align}

\section{Proof of Theorem \ref{main:1}} \label{main_proof}
\subsection{Preliminary lemmas} 
To prove Theorem \ref{main:1}, we first prove Lemma \ref{Lipschitz}. 
The following lemma shows that it suffices to prove Lemma \ref{Lipschitz} in the case $c=1$.

\begin{lem} \label{reduced}
The proof of Lemma \ref{Lipschitz} is reduced to the case $\sigma(t)=t^p$ (i.e., $c=1$). 
\end{lem}

\begin{proof}
Let $B^{\circ}_n$ denote the coefficients of $\widehat{G}^{(\sigma)}_{a,a;2a}(t)$ in the case $c=1$. 
Then, we have 
\begin{align*}
\dfrac{B_n}{A_n} = \dfrac{B_n^{\circ}}{A_n} +(-1)^{l+1} B_p(a, a)^{-1} \left(\dfrac{c^{\frac{n+a}{p}} -1}{n+a} \right) \dfrac{A^{(1)}_{\frac{n-l}{p}}}{A_n}. 
\end{align*}
Here, the value $(c^{\frac{m}{p}}-1)/m=0$ when $p \nmid m$. 
By \cite[Corollary 4.3]{N2}, if $n_1 \equiv n_2 \pmod{p^n}$, then we have 
$$\dfrac{A^{(1)}_{\frac{n_1-l}{p}}}{A_{n_1}} \equiv \dfrac{A^{(1)}_{\frac{n_2-l}{p}}}{A_{n_2}} \pmod{p^n}. $$
Therefore, it suffices to show that for $n_1 \equiv n_2 \pmod{p^m}$, 
$$\dfrac{c^{\frac{n_1}p}-1}{n_1} \equiv \dfrac{c^{\frac{n_2}p}-1}{n_2} \pmod{p^m}. $$
This is proved in the proof of \cite[Lemma 3.7]{Wang}. 
\end{proof}

Throughout the rest of this subsection, we assume that $c=1$.

\begin{lem} \label{Beta_cong}
Let $k$ be a positive integer such that $a+k=dp^m$ with $p \nmid d$.  
Then we have  
\begin{align*}
B_p(a, a)^{-1} \dfrac{\{2a\}_k}{\{1\}_k} \equiv (-1)^l \pmod{p^{2m}}. 
\end{align*}
\end{lem}

\begin{proof}
For a $p$-adic integer 
$$\al=-[\al]_0 -[\al]_1p -[\al]_2 p^2 -\cdots, $$
we write $\overline{\al}:=-[\al]_0 -[\al]_1p -[\al]_2 p^2 -[\al]_{m-1}p^{m-1}$.  
We compute that  
 \begin{align*}
 &B_p(a, a)^{-1} \dfrac{\{2a\}_k}{\{1\}_k} \\
 &= B_p(a, a)^{-1}\dfrac{\{2dp^m-2k\}_k}{\{1\}_k} \\
&= \dfrac{\Gamma_p(2a)}{\Gamma_p(a)^2} \dfrac{\{2dp^m -2k\}_k }{\{1\}_k} \\
&\overset{(*)}{=}-(-1)^{2l} \dfrac{\Gamma_p(1-a)^2}{\Gamma_p(1-2a)} \dfrac{\{2dp^m -2k\}_k }{\{1\}_k} \\
&=- \dfrac{\Gamma_p(1+k-dp^m)^2}{\Gamma_p(1+2k-2dp^m)} \dfrac{\{2dp^m -2k\}_k }{\{1\}_k} \\
& \overset{(**)}{\equiv} (-1)^{2k-\lfloor \frac{2k}p \rfloor - k+ \lfloor \frac{k}p \rfloor } \dfrac{\{1\}_{k-\od p^m}^2}{\{1\}_{2k-2\od p^m}} \dfrac{(2k-2\od p^m) \cdots (k+1 -2\od p^m)}{\{1\}_k} \pmod{p^{2m}}\\
&=(-1)^{2k-\lfloor \frac{2k}p \rfloor - k+ \lfloor \frac{k}p \rfloor }  \dfrac{\{1\}_{k-\od p^m}^2}{\{1\}_{k-2\od p^m}} \dfrac{1}{\{1\}_k} \\
&=(-1)^{2k-\lfloor \frac{2k}p \rfloor - k+ \lfloor \frac{k}p \rfloor }   \dfrac{\{1\}_k\{1+\od p^m\}_{-1-\od p^m}}{\{1\}_k \{1\}_{-1-\od p^m}} \prod_{\substack{j=1+\od p^m\\  p \nmid j}}^k \dfrac{\left(1-\frac{\od p^m}j\right)^2}{\left(1-\frac{2\od p^m}j\right)}, 
 \end{align*} 
 where $(*)$ follows from the reflection formula \eqref{reflection} and $(**)$ follows from \eqref{p_adic_gamma}.  
We compute that 
$$\dfrac{\left(1-\frac{\od p^m}j\right)^2}{\left(1-\frac{2\od p^m}j\right)} \equiv \left(1-2\frac{\od p^m}j\right)\left(1+2\frac{\od p^m}j\right)\equiv  1 \pmod{p^{2m}}$$
and
$$ \dfrac{\{1+\od p^m\}_{-1-\od p^m}}{ \{1\}_{-1-\od p^m}} = (-1)^{-1-\od p^m - \lfloor \frac{-1-\od p^m } {p} \rfloor}  = (-1)^{-\od p^{m-1}(p-1)} = 1. $$
On the other hand, since $k \equiv l \pmod{p}$, we write $k=l+bp$. 
Since $p$ is odd, we have   
\begin{align*}
2k - \lfloor \frac{2k}p \rfloor -k + \lfloor \frac{k}p \rfloor  &= 2(l+bp) - \lfloor \frac{2(l+bp)}p \rfloor -(l+bp)+\lfloor \frac{l+bp}p \rfloor  \\
& \equiv  l+ b(p-1) - \lfloor \frac{2l}{p} \rfloor  \pmod{2} \\
&\equiv l - \lfloor \frac{2l}{p} \rfloor  \pmod{2}. 
\end{align*}
By the assumption, we have $2l=2[a]_0 \in \{0, \ldots, p-1\}$. 
Therefore, we obtain $\lfloor 2l /p \rfloor=0$, thus we conclude that  
$$2k-\lfloor \frac{2k}p \rfloor - k+ \lfloor \frac{k}p \rfloor \equiv l \pmod{2}. $$
Therefore, we have 
\begin{align*} 
B_p(a, a)^{-1} \dfrac{\{2a\}_k}{\{1\}_k} \equiv (-1)^l \pmod{p^{2m}}, 
\end{align*}
which finishes the proof. 
\end{proof}

\begin{lem} \label{psi_p}
If $a+l=cp^n$ with $p \nmid c$, then 
we have 
$$\dfrac{B_l}{A_l} \equiv 2\psi_p(a+l) - 2\psi_p(1+l) \pmod{p^n}, $$
where $\psi_p(z)$ is the $p$-adic digamma function defined in Definition \ref{definition:3}. 
\end{lem}

\begin{proof}
We have 
 \begin{align*}
 \dfrac{B_l}{A_l} &=\dfrac{1}{a+l} \left(1- (-1)^l B_p(a,a)^{-1} \dfrac1{A_l} \right) \\
 &= \dfrac1{cp^n}  \left(1- (-1)^l B_p(a,a)^{-1} \dfrac{(2a)_l l!}{(a)_l(a)_l} \right). 
 \end{align*}
Rearranging this identity, we obtain 
 \begin{align*}
 (-1)^l \left( 1 -cp^n \dfrac{B_l}{A_l} \right) &=B_p(a, a)^{-1} \dfrac{(2a)_l l!}{(cp^n-l)_l^2} \\
 &=B_p(a, a)^{-1} \dfrac{(2a)_l}{l!} \dfrac{(l!)^2}{\{(l-cp^n) \cdots (1-cp^n)\}^2}. 
 \end{align*}
By \cite[Lemma 3.9]{Wang}, we have
 $$\dfrac{l!}{(l-cp^n) \cdots (1-cp^n)} \equiv (-1)^l(1+cp^n(\psi_p(1+l) + \gamma_p)) \pmod{p^{2n}}. $$
 On the other hand, by Lemma \ref{Beta_cong}, we have 
\begin{align*} 
B_p(a, a)^{-1} \dfrac{(2a)_l}{l!} \equiv (-1)^l \pmod{p^{2n}}. 
\end{align*}
Therefore, 
we have 
$$ (-1)^l \left( 1 -cp^n \dfrac{B_l}{A_l} \right) \equiv (-1)^{l}(1+cp^n(\psi_p(1+l) + \gamma_p))^2 \pmod{p^{2n}}.  $$
Since we have $\psi_p(a+l) \equiv \psi_p(0)=-\gamma_p \pmod{p^n}$ (see \cite[(2.13) and Theorem 2.6 (1)]{As}), 
we obtain  
$$\dfrac{B_l}{A_l} \equiv -(2 \psi_p (1+ l) + 2\gamma_p ) \equiv 2\psi_p(a+l) - 2\psi_p(1+l) \pmod{p^n}, $$
which finishes the proof. 
\end{proof}

\begin{lem} \label{lem:1}
If $k= l+bp^m$ with $p \nmid b$, then we have 
$$\dfrac{B_k}{A_k} \equiv \dfrac{B_l }{A_l} \pmod{p^m}. $$
\end{lem}

\begin{proof}
Again, we write $a+l =cp^n$ with $p \nmid c$. 

\fbox{Case I: $m \neq n$} 

We have 
\begin{align*}
&1-(a+l+bp^m) \dfrac{B_{l+bp^m}}{A_{l+bp^m}} \\
&=(-1)^l B_p(a, a)^{-1} \dfrac{A^{(1)}_{bp^{m-1}}}{A_{l+bp^m}} \\
&=(-1)^l B_p(a, a)^{-1} \left(\dfrac{\{1\}_{l+bp^m}}{\{a\}_{l+bp^m}} \right)^2 \dfrac{\{2a\}_{l+bp^m}}{\{1\}_{l+bp^m}} \\
& = (-1)^l B_p(a, a)^{-1} \dfrac1{A_l} \left(\dfrac{\{1+l\}_{bp^m}}{\{a+l\}_{bp^m}} \right)^2 \left(\dfrac{\{2a+l\}_{bp^m}}{\{1+l\}_{bp^m}} \right)   \\
&\overset{(*)}{\equiv} (-1)^l B_p(a, a)^{-1} \dfrac1{A_l} \left(1-bp^m(\psi_p(a+l) -\psi_p(1+l)) \right)^2 \left(1-bp^m(\psi_p(1+l) -\psi_p(2a+l)) \right) \pmod{p^{2m}}\\
& \equiv \left(1-(a+l) \dfrac{B_l}{A_l} \right) \left(1-2bp^m(\psi_p(a+l) -\psi_p(1+l)) \right)\left(1-bp^m(\psi_p(1+l) -\psi_p(2a+l)) \right) \pmod{p^{2m}}, 
\end{align*}
where $(*)$ follows from \cite[Lemma 4.5]{N2}. 
Since $2a+l \equiv -l \pmod{p^n}$ and $\psi_p(1+l) =\psi_p(-l)$ (see \cite[Theorem 2.6 (2)]{As}), we have 
$$\psi_p(1+l) -\psi_p(2a+l) \equiv 0 \pmod{p^n}. $$
Therefore, by rearranging the equation above, if $n>m$ (resp. $n<m$), we have 
\begin{align*}
1-(a+l+bp^m) \dfrac{B_{l+bp^m}}{A_{l+bp^m}} \equiv 1-(a+l) \dfrac{B_l}{A_l} -2bp^m(\psi_p(a+l) -\psi_p(1+l))
\end{align*}
modulo $p^{2m}$ (resp. $p^{m+n}$).  
Hence we obtain 
\begin{align}\label{eq:1}
(a+l+bp^m) \left(\dfrac{B_l}{A_l} - \dfrac{B_{l+bp^m}}{A_{l+bp^m}} \right) \equiv bp^m \left( \dfrac{B_l}{A_l} -2 (\psi_p(a+l) -\psi_p(1+l)) \right) &
\end{align}
modulo $p^{2m}$ (resp. $p^{m+n}$). 
By Lemma \ref{psi_p}, the right-hand side of \eqref{eq:1} is congruent to $0$ modulo $p^{2m}$ (resp. $p^{m+n}$), hence we have 
\begin{align*}
(a+l+bp^m) \left(\dfrac{B_l}{A_l} - \dfrac{B_{l+bp^m}}{A_{l+bp^m}} \right) \equiv 0 
\end{align*}
modulo $p^{2m}$ (resp. $p^{m+n}$). 
This implies 
$$ \dfrac{B_{l+bp^m}}{A_{l+bp^m}} \equiv  \dfrac{B_{l}}{A_{l}}  \pmod{p^m}$$
since $\operatorname{ord}_p(a + l + bp^m)=m$ (resp. $\operatorname{ord}_p(a + l + bp^m)=n$). 
This proves Case I.

\bigskip

\fbox{Case II: $m=n$} 

Write $a+l=cp^m$, $k=l+bp^m$, and $a+k=(b+c)p^m=d p^{m+m'}$ with $p \nmid d$. 
We may assume $m' \ge 1$, since the case $m'=0$ can be handled by the same argument as in Case I. 
Then 
\begin{align} \label{numerator} 
\dfrac{B_{l+bp^m}}{A_{l+bp^m}} -\dfrac{B_l}{A_l} = \dfrac{c[1-(-1)^l B_p(a, a)^{-1}\frac{A^{(1)}_{bp^{m-1}}}{A_{bp^m+l}}] - dp^{m'} [1- (-1)^l B_p(a,a)^{-1}\frac1{A_l}]}{cdp^{m+m'}}. 
\end{align}
We claim that 
$$c[1-(-1)^l B_p(a, a)^{-1}\frac{A^{(1)}_{bp^{m-1}}}{A_{bp^m+l}}] - dp^{m'} [1- (-1)^l B_p(a,a)^{-1}\frac1{A_l}]\equiv 0 \pmod{p^{2m+m'}}. $$
First, we calculate $B_p(a, a)^{-1}A^{(1)}_{bp^{m-1}}/A_{bp^m+l}$. We have 
\begin{align*}
B_p(a, a)^{-1} \dfrac{A^{(1)}_{bp^{m-1}}}{A_{bp^m+l}}= B_p(a, a)^{-1}\dfrac{\{1\}_{l+bp^m} \{2a\}_{l+bp^m}}{\{a\}_{l+bp^m}^2} = B_p(a, a)^{-1}\dfrac{1}{A_l} \left(\dfrac{\{1+l\}_{bp^m}}{\{a+l\}_{bp^m}} \right)^2 \dfrac{\{2a+l\}_{bp^m}}{\{1+l\}_{bp^m}}. 
\end{align*}
As in the proof of \cite[Lemma 3.10, Case II (1)]{Wang}, we have 
\begin{align*}
\dfrac{\{1+l\}_{bp^m}}{\{a+l\}_{bp^m}} 
&\equiv
\dfrac{(-1)^l  \displaystyle \prod_{0 \le i \le l-1} (a + i -dp^{m+m'})}{l!} \pmod{p^{2m+m'}} \\
& \equiv 
\dfrac{(-1)^l}{l!} \left[\prod_{i=0}^{l-1} (a+i) -dp^{m+m'} (a)_l \sum_{j=0}^{l-1} \dfrac1{a+j} \right] \pmod{p^{2m+m'}} \\
& \equiv 
(-1)^l \left[\dfrac{(a)_l}{l!} + (-1)^ldp^{m+m'} \sum_{j=1}^{l} \dfrac1j \right] \pmod{p^{2m+m'}}. 
\end{align*}
Applying Lemma \ref{Beta_cong} to $k=l+bp^m$, we obtain 
\begin{align*} 
B_p(a, a)^{-1} \dfrac{\{2a+l\}_{bp^m}}{\{1+ l\}_{bp^m}}\equiv (-1)^{l} \dfrac{l!}{(2a)_l} \pmod{p^{2(m+m')}}.  
\end{align*}
Since we have $(a)_l/l! \equiv (-1)^l \pmod{p^n}$, we conclude that 
\begin{align} \label{cong:1}
B_p(a, a)^{-1} \dfrac{A^{(1)}_{bp^{m-1}}}{A_{bp^m+l}} \equiv (-1)^l \left(1+2dp^{m+m'} \sum_{j=1}^l \dfrac1j \right) \pmod{p^{2m+m'}}. 
\end{align}

Secondly, we calculate $B_p(a, a)^{-1}/A_l$. 
By \cite[Lemma 3.9]{Wang}, we have 
$$\left( \dfrac{l!}{(a)_l} \right)^2 \equiv 1+ 2cp^m(\psi_p(1+l) + \gamma_p)=1+ 2cp^m \sum_{j=1}^l \frac{1}j \pmod{p^{2m}}. $$
On the other hand, by Lemma \ref{Beta_cong}, we have 
$$B_p(a, a)^{-1} \dfrac{(2a)_l}{l!} \equiv (-1)^{l} \pmod{p^{2m}}. $$
Therefore, we conclude that 
\begin{align} \label{cong:2}
B_p(a, a)^{-1} \dfrac1{A_l} \equiv (-1)^l \left(1+2cp^m \sum_{j=1}^l \dfrac1j \right) \pmod{p^{2m}}. 
\end{align}
By \eqref{cong:1} and \eqref{cong:2}, the numerator in \eqref{numerator} satisfies
\begin{align*}
&c[1-(-1)^l B_p(a, a)^{-1} \dfrac{A^{(1)}_{bp^{m-1}}}{A_{bp^m+l}}] -dp^{m'} [1 - (-1)^l B_p(a, a)^{-1} \dfrac1{A_l}] \\
&=-b+(-1)^l \left(B_p(a, a)^{-1} \dfrac{dp^{m'}}{A_l} -cB_p(a, a)^{-1} \dfrac{A^{(1)}_{bp^{m-1}}}{A_{bp^m+l}} \right) \\
&\equiv -b + \left( dp^{m'}\left(1+2cp^m \sum_{j=1}^l \dfrac1j \right) -c\left(1+2dp^{m+m'} \sum_{j=1}^l \dfrac1j \right) \right) \pmod{p^{2m+m'}}\\
& = -b+b =0. 
\end{align*}
This proves Case II. 
\end{proof}

\begin{lem} \label{Lipschitz}
If $k \equiv k' \pmod{p^{s}}$, then 
$$\dfrac{B_k}{A_k} \equiv \dfrac{B_{k'}}{A_{k'}} \pmod{p^{s}}. $$
\end{lem}
\begin{proof}
If $k \not \equiv l \pmod{p}$, then 
$$\dfrac{B_k}{A_k}=\dfrac{1}{k+a}, $$
so the result follows. 
Suppose $k \equiv l \pmod{p}$. 
We write $a+l=cp^n$, $k=l+bp^m$ with $p \nmid bc$. 
It is enough to show the lemma for the case $k'=k+p^{s}$ for any $s \in \Z_{\ge 1}$.  
\bigskip

\fbox{Case I: $s \le m$} 

By Lemma \ref{lem:1}, both ratios are congruent to $B_l/A_l \pmod{p^{s}}$. Hence, the assertion follows. 
\\ 

\fbox{Case II: $s > m$, and either $m \neq n$, or $m =n$ with $\operatorname{ord}_p(k+a)=m$} 

We have 
\begin{align*}
&1-(k'+a) \dfrac{B_{k'}}{A_{k'}} \\
&=(-1)^l B_p(a, a)^{-1}\dfrac{A^{(1)}_{\lfloor \frac{k'}p \rfloor}}{A_{k'}} \\
&=(-1)^l B_p(a, a)^{-1}\dfrac{\{1\}_{l+bp^m + p^{s}}\{2a\}_{l+bp^m + p^{s}}}{{\{a\}^2_{l+bp^m + p^{s }}}} \\
&=(-1)^l B_p(a, a)^{-1}\dfrac{\{1\}_{l+bp^m }\{2a\}_{l+bp^m }}{{\{a\}^2_{l+bp^m}}} \dfrac{\{1+l+bp^m\}_{p^{s } }\{2a+l+bp^m\}_{p^{s}}}{{\{a+l+bp^m\}^2_{p^{s}}}}\\
& \equiv \left(1-(k+a) \dfrac{B_k}{A_k} \right) (1+p^{s} (\psi_p (1+l+bp^m) -\psi_p(a+l+bp^m)))^2 \\
& \times (1+p^{s}(\psi_p (2a+l+bp^m) -\psi_p(1+l+bp^m))) \pmod{p^{2s}} \\
& \equiv \left(1-(k+a) \dfrac{B_k}{A_k} \right)\left(1-p^{s}\dfrac{B_l}{A_l} \right) \pmod{p^{n^*+s}},  
\end{align*}
where $n^*=\min \{n, m\}$. 
Therefore, it follows that 
\begin{align*}
&1-(k+a+p^{s}) \dfrac{B_{k'}}{A_{k'}} 
\equiv 1- (k+a) \dfrac{B_{k}}{A_{k}} -p^{s} \dfrac{B_l}{A_l} \pmod{p^{n^*+s}},  
\end{align*}
hence we have 
\begin{align*}
(k+a) \left(\dfrac{B_k}{A_k} -\dfrac{B_{k'}}{A_{k'}} \right) \equiv p^{s}  \left(\dfrac{B_{k'}}{A_{k'}} -\dfrac{B_{l}}{A_{l}} \right) \equiv 0 \pmod{p^{n^*+s}}. 
\end{align*}
By the assumption in Case II, we have $\operatorname{ord}_p(k+a)=n^*$. Hence, we have 
$$\dfrac{B_k}{A_k} \equiv \dfrac{B_{k'}}{A_{k'}} \pmod{p^{s}}. $$
This proves Case II. 

\bigskip 

\fbox{Case III: $s > m$ and $m =n$ with $\operatorname{ord}_p{(k+a)} > m$} 

In this case, we write $a+l=cp^m$, $k=l+bp^m$, $k+a=(b+c)p^m=dp^{m'+m}$ and $k'=k+p^{s}$ with $s \ge 1$, $p \nmid bcd$.  
Then, 
\begin{align} \label{numerator2}
\dfrac{B_{k'}}{A_{k'}} - \dfrac{B_k}{A_k} = \dfrac{-p^{s} +(-1)^l B_p(a, a)^{-1} [(dp^{m+m'} + p^{s}) \frac{A^{(1)}_{bp^{m-1}}}{A_{l+bp^m}} - (dp^{m+m'}) \frac{A^{(1)}_{bp^{m-1}+p^{s-1}}}{A_{l+bp^m+p^{s}}}]}{(dp^{m+m'} +p^{s}) dp^{m+m'}}. 
\end{align}
We claim that the numerator of \eqref{numerator2} is congruent to $0$ modulo $(dp^{m+m'}+p^{s})p^{m+m'} p^{s}$. 
For simplicity, we denote the numerator of \eqref{numerator2} by $(*)$.

First, we consider the case $s \le m+m'$. 
As in the proof of \cite[Lemma 3.5, Case III (1)]{Wang}, we have 
\begin{align*}
\dfrac{\{1\}_{l+bp^m+p^{s}}}{\{a\}_{l+bp^m+p^{s}}} 
\equiv 
(-1)^l \left[1+(dp^{m'+m}+p^{s}) (\psi_p(1+l+bp^m+p^{s}) + \gamma_p) \right] \pmod{(dp^{m+m'}+p^{s})p^{s}}. 
\end{align*}
By Lemma \ref{Beta_cong}, we have 
$$B_p(a, a)^{-1} \dfrac{\{2a\}_{l+bp^m+p^{s}}}{\{1\}_{l+bp^m+p^{s}}} \equiv (-1)^l \pmod{(dp^{m+m'}+p^{s})^2}. $$
We conclude that 
\begin{align*}
&B_p(a, a)^{-1}\dfrac{A^{(1)}_{bp^{m-1}+p^{s-1}}}{A_{l+bp^m+p^{s}}} \\
&\quad \equiv (-1)^l \left[ 1+2(dp^{m+m'}+p^{s}) (\psi_p(1+l+bp^m+p^{s}) + \gamma_p) \right] \pmod{(dp^{m+m'}+p^{s})p^{s}}. 
\end{align*}
On the other hand, 
As in the proof of \cite[Lemma 3.5, Case III (1)]{Wang}, we have 
\begin{align*}
\dfrac{\{1\}_{l+bp^m}}{\{a\}_{l+bp^m}} \equiv (-1)^l \left[1+dp^{m+m'} (\psi_p(1+l+bp^m) + \gamma_p) \right] \pmod{p^{m+m'+s}}. 
\end{align*}
Using Lemma \ref{Beta_cong} 
\begin{align*}
B_p(a, a)^{-1}\dfrac{\{2a\}_{l+bp^m}}{\{1\}_{l+bp^m}} \equiv (-1)^l \pmod{p^{m+m'+s}}, 
\end{align*}
we obtain 
\begin{align*}
B_p(a, a)^{-1}\dfrac{A^{(1)}_{bp^{m-1}}}{A_{l+bp^m}} \equiv (-1)^l \left[1+2dp^{m+m'} (\psi_p(1+l+bp^m) + \gamma_p) \right]\pmod{p^{m+m'+s}}.  
\end{align*}
Hence, $(*)$ is congruent to 
\begin{align*}
-p^{s}+ \left(p^{s} +2(dp^{m+m'}+p^{s}) dp^{m+m'} \left( \psi_p(1+l+bp^m) - \psi_p(1+l+bp^m+p^{s}) \right) \right) 
\end{align*}
modulo $(dp^{m+m'}+p^{s})p^{m+m'+s}$. 
Since 
$$\psi_p(1+l+bp^m) \equiv \psi_p(1+l+bp^m+p^{s}) \pmod{p^{s}}, $$
$(*)$ is congruent to $-p^{s}+p^{s}=0$. 
This proves the first case.

Secondly, we consider the case $s > m+m'$. 
As in the proof of \cite[Lemma 3.5, Case III (2)]{Wang}, we have 
\begin{align*}
&\dfrac{\{1\}_{l+bp^m+p^{s}}}{\{a\}_{l+bp^m+p^{s}}} \equiv \dfrac{\{1\}_{l+bp^m}}{\{a\}_{l+bp^m}} \left(1+ p^{s} \left(\psi_p(1+l+bp^m) + \gamma_p \right) \right) \pmod{p^{m+m'+s}}. 
\end{align*}
Similarly, we have 
\begin{align*}
&\dfrac{\{2a\}_{l+bp^m+p^{s}}}{\{1\}_{l+bp^m+p^{s}}} \\
& \equiv \dfrac{\{2a\}_{l+bp^m}}{\{1\}_{l+bp^m}}  \left(1+ p^{s} \left(\psi_p(2a+l+bp^m) + \gamma_p \right) \right)\left(1- p^{s} \left(\psi_p(1+l+bp^m) + \gamma_p \right) \right) \pmod{p^{2{s}}}. 
\end{align*} 
Since we have 
\begin{align*}
\psi_p(2a+l+bp^m) = \psi_p(2dp^{m+m'} -k) \equiv \psi_p(-k) \pmod{p^{m+m'}}, 
\end{align*}
and by \cite[Theorem 2.6 (2)]{As}, 
\begin{align*}
\psi_p(1+l+bp^m) = \psi_p(1+k) = \psi_p(-k), 
\end{align*}
we obtain 
\begin{align*}
\dfrac{\{2a\}_{l+bp^m+p^{s}}}{\{1\}_{l+bp^m+p^{s}}} \equiv \dfrac{\{2a\}_{l+bp^m}}{\{1\}_{l+bp^m}} \pmod{p^{m+m'+s}}. 
\end{align*}
Substituting this result into $(*)$, we have 
\begin{align*}
(*) \equiv -p^{s} +(-1)^l B_p(a, a)^{-1}\dfrac{A^{(1)}_{bp^{m-1}}}{A_{l+bp^m}} p^{s} \left( 1- 2dp^{m+m'} \left( \psi_p(1+l+bp^m) + \gamma_p \right) \right) \pmod{p^{2m+s+2m'}}. 
\end{align*}
As in the proof of \cite[Lemma 3.5, Case III (2)]{Wang}, we obtain 
\begin{align*}
\dfrac{\{1\}_{l+bp^m}}{\{a\}_{l+bp^m}} \equiv (-1)^l\left(1+dp^{m+m'} \left( \psi_p(1+l+bp^m) + \gamma_p \right) \right) \pmod{p^{2(m+m')}}
\end{align*}
and using Lemma \ref{Beta_cong}, 
\begin{align*}
B_p(a, a)^{-1} \dfrac{\{2a\}_{l+bp^m}}{\{1\}_{l+bp^m}} \equiv (-1)^l \pmod{p^{2(m+m')}}, 
\end{align*}
hence we have 
\begin{align*}
(*) &\equiv -p^{s} + p^{s} \left(1+2dp^{m+m'} \left( \psi_p(1+l+bp^m)+\gamma_p \right) \right) \left(1 -2dp^{m+m'} \left( \psi_p(1+l+bp^m)+\gamma_p \right) \right) \\
& \equiv 0 \pmod{p^{2m+s+2m'}}. 
\end{align*}
Hence, again we obtain 
$$\dfrac{B_k}{A_k} \equiv \dfrac{B_{k'}}{A_{k'}} \pmod{p^s}. $$
Now, we prove that $B_k/A_k \in W$. For $k \not \equiv l \pmod{p}$, $B_k/A_k =1/(k+a) \in W$. 
For $k \equiv l \pmod{p}$, it follows from $A_l \in \Z_p^{\times}$ and 
$$\dfrac{B_k}{A_k} \equiv \dfrac{B_l}{A_l} \pmod{p}. $$ 
This completes the proof. 
\end{proof}

\subsection{Proof of Theorem \ref{main:1}}
We finish the proof of Theorem \ref{main:1}. 
Put 
$$S_m :=\sum_{i+j=m} A_{i+p^n} B_j - A_i B_{j+p^n}.  $$
We note that the statement of Theorem \ref{main:1} is equivalent to 
$$S_m \equiv 0 \pmod{p^n}$$
for all $m \ge 0$.

\begin{lem} \label{A_cong}
We have 
$$S_m \equiv \sum_{i+j=m} (A_{i+p^n} A_j-A_i A_{j+p^n}) \dfrac{B_j}{A_j} \pmod{p^n}. $$
\end{lem}

\begin{proof}
We compute 
\begin{align*}
S_m & = \sum_{i+j=m} A_{i+p^n} B_j -A_i A_{j+p^n} \dfrac{B_{j+p^n}}{A_{j+p^n}} \\
& \overset{(*)}{\equiv} \sum_{i+j=m} A_{i+p^n} B_j -A_i A_{j+p^n} \dfrac{B_{j}}{A_{j}} \pmod{p^n} \\
& = \sum_{i+j=m} (A_{i+p^n} A_j -A_i A_{j+p^n}) \dfrac{B_{j}}{A_{j}},  
\end{align*}
where $(*)$ follows from Lemma \ref{Lipschitz}. 
\end{proof}

\begin{lem}[{\cite[Lemma 4.8]{N2}}] \label{keylem}
For all $m, k, n \in \Z_{\ge 0}$ and $0 \le l \le n$, we have 
$$\sum_{\substack{i+j=m \\ i \equiv k \mod{p^{n-l}}}}(A_i A_{j+p^n} - A_j A_{i+p^n}) \equiv 0 \pmod{p^{l+1}}. $$
\end{lem}

\begin{proof}[Proof of Theorem \ref{main:1}]
We will show 
$$S_m \equiv 0 \pmod{p^n}$$
for all $m \ge 0$. 
Put $q_k=B_k/A_k$. 
By Lemmas \ref{A_cong} and \ref{Lipschitz}, we have
$$S_m \equiv \sum_{k=0}^{p^n-1}q_k \overbrace{\sum_{\substack{i+j=m\\j\equiv k \ {\rm mod} \ p^n}} (A_{i+p^n}A_j-A_iA_{j+p^n})}^{(*)} \pmod{p^n}.$$
It follows from Lemma \ref{keylem} that  $(*)$ is zero modulo ${p}$. Therefore, again by Lemma \ref{Lipschitz}, one can rewrite  
$$S_m \equiv \sum_{k=0}^{p^{n-1}-1}q_k \overbrace{\sum_{\substack{i+j=m\\j\equiv k \ {\rm mod} \ p^{n-1}}} (A_{i+p^n}A_j-A_iA_{j+p^n})}^{(**)} \pmod{p^n}.$$
It follows from Lemma \ref{keylem} that  $(**)$ is zero modulo ${p^2}$. Therefore, again by Lemma \ref{Lipschitz}, one can rewrite  
 $$S_m \equiv \sum_{k=0}^{p^{n-2}-1}q_k \sum_{\substack{i+j=m\\j\equiv k \ {\rm mod} \ p^{n-2}}} (A_{i+p^n}A_j-A_iA_{j+p^n}) \pmod{p^n}.$$
 Continuing the same discussion, one finally obtains 
$$S_m \equiv \sum_{i+j=m} (A_{i+p^n}A_j-A_iA_{j+p^n}) =0 \pmod{p^n}, $$
which finishes the proof. 
\end{proof}

\section{syntomic regulator of the Hesse cubic curve} \label{syntomic_reg_Hesse}
Note that $a=1/3$ satisfies the conditions in Section \ref{defFhat} if and only if $p \equiv 1 \pmod{3}$. 
Hence, we suppose that $p \equiv 1 \pmod{3}$ in this section.    
Let $W=W(\F)$ be the Witt ring of $\F$ and $K=\operatorname{Frac}(W)$ the fractional field of $W$.

\subsection{Hesse cubic curve}
Let $f_{K} \colon \overline{X}_K \to \mathbb{P}_{K}^1$ be the Hesse family of cubic curves over $K$, 
whose fiber $X_{t, K}:=f_{K}^{-1}(t)$ over $t \in \mathbb{A}^1_{K}$ is defined by 
$$x_0^3+y_0^3+z_0^3-3tx_0y_0z_0=0. $$
This is smooth over $S_{K}:= \mathbb{A}^1_{K} \setminus \mu_3$, where $\mu_3 \subset K$ denotes the group of cubic roots of unity. We fix a primitive cubic root of unity $\zeta \in \mu_3$. 
The affine equation is written as 
$$x^3+y^3+1-3txy=0 \qquad (x=x_0/z_0, y=y_0/z_0). $$
There is a projective flat morphism 
$$\overline{f} \colon \overline{Y} \to \mathbb{P}^1_{W}$$
over $W$ extending the fibration $f_{K}$ such that the following conditions hold: 
\begin{itemize}
\item $\overline{Y}$ is a smooth projective scheme over $W$. 
\item The singular fiber $\of^{-1}(t)$ $(t \in \{1, \zeta, \zeta^2, \infty\})$ is a N\'eron $3$-gon. 
\end{itemize}
Put $S:=\mathbb{P}^1_{W} \setminus \{1, \zeta, \zeta^2, \infty\}$, $X:=\of^{-1}(S)$ and $X_K := X \times_{W} K$.

\subsection{Gauss-Manin connection} 
\begin{dfn} \label{diff_form}
We define differential forms on $X$ by 
\begin{align*}
&\omega_t = \dfrac{dy}{x^2-ty} = \dfrac{-dx}{y^2-tx}, \\
&\eta_t= xy \omega_t = \dfrac{xydy}{x^2-ty} = -\dfrac{xydx}{y^2-tx}. 
\end{align*}
\end{dfn}
Since $\omega_t$ (resp. $\eta_t$) is of the first (resp. second) kind, they define cohomology classes, and we denote them by the same letter.

Put 
\begin{align*}
&\widetilde{\omega}_t := \dfrac1{F_{\frac13, \frac13; \frac23}(t^3)} \omega_t, \\
&\widetilde{\eta}_t := (-t^2 F_{\frac13, \frac13; \frac23}(t^3)+(1-t^3) [F_{\frac13, \frac13; \frac23}(t^3)]') \omega_t + F_{\frac13, \frac13; \frac23}(t^3) \eta_t, 
\end{align*}
where $[F_{\frac13, \frac13; \frac23}(t^3)]'$ denotes $\frac{d}{dt}F_{\frac13, \frac13; \frac23}(t^3)$. 

\begin{prop} \label{GM}
Let $\nabla \colon H^1_{\rm dR} (X_{K}/S_{K}) \to \mathscr{O}(S_{K}) dt \otimes H^1_{\rm dR} (X_{K}/S_{K}) $ be the Gauss-Manin connection. 
It naturally extends on $K[[t]] \otimes_{\mathscr{O}(S_{K})} H^1_{\rm dR} (X_{K}/S_{K})$ which we also write by $\nabla$. Then 
\begin{align}
&\nabla
\begin{pmatrix}
\omega_t& \eta_t
\end{pmatrix}
=\dfrac{dt}{1-t^3}
\otimes 
\begin{pmatrix}
\omega_t & \eta_t
\end{pmatrix}
\begin{pmatrix}
t^2 & t \\
-1& -t^2
\end{pmatrix}, \label{GM1} \\
&\nabla
\begin{pmatrix}
\widetilde{\omega}_t& \widetilde{\eta}_t
\end{pmatrix}
=dt
\otimes 
\begin{pmatrix}
\widetilde{\omega}_t & \widetilde{\eta}_t
\end{pmatrix}
\begin{pmatrix}
0 & 0 \\
-(1-t^3)^{-1}F_{\frac13, \frac13; \frac23}(t^3)^{-2} & 0
\end{pmatrix}. \label{GM2}
\end{align}
\end{prop}

\begin{proof}
The first equation \eqref{GM1} is proved in \cite[Corollary 2.5]{N}, and the second equation \eqref{GM2} follows from \eqref{GM1} and the Picard-Fuchs equation 
$$\left((1-t^3)\dfrac{d^2}{dt^2} - 3t^2 \dfrac{d}{dt}-t \right)F_{\frac13, \frac13; \frac23}(t^3)=0. $$
\end{proof}

\subsection{Frobenius action}
Let $\sigma$ be the $p$th Frobenius on $W[t, (1-t^3)^{-1}]^{\dagger}$, the ring of overconvergent power series defined by $\sigma(t)=ct^p$ with $c \in 1+ p W$, which extends to $\mathscr{O}(S_{K})^{\dagger}=K[t, (1-t^3)^{-1}]^{\dagger}:=K \otimes W[t, (1-t^3)^{-1}]^{\dagger}$. 
Put $X_{\F}:=X \times _{W} \F$ and $S_{\F}:=S \times_{W} \F$. 
Then, there is a comparison isomorphism with algebraic de Rham cohomology, 
$$c \colon H^i_{\rm rig}(X_{\F}/S_{\F}) \simeq H^i_{\rm dR}(X_{K}/S_{K}) \otimes_{\mathscr{O}(S_{K})} \mathscr{O}(S_{K})^{\dagger}. $$
Since a $p$th Frobenius $\Phi_{X/S, \sigma}$ on $H^i_{\rm rig}(X_{\F}/S_{\F})$ (depending on $\sigma$) is defined, $\Phi_{X/S, \sigma}$ acts on $H^i_{\rm dR}(X_{K}/S_{K}) \otimes_{\mathscr{O}(S_{K})} \mathscr{O}(S_{K})^{\dagger}$ via the above comparison.

\begin{prop} \label{Frobenius}
 Suppose that $\sigma$ is given by $\sigma(t)=ct^p$ with $c \in 1+pW$.  
 Then we have 
 $$\Phi_{X/S, \sigma} (\widetilde{\eta}_t) \in K \widetilde{\eta}_t, \qquad \Phi_{X/S, \sigma} (\widetilde{\omega}_t) \equiv p B_p\left(\dfrac13, \dfrac13 \right)^{-1}\widetilde{\omega}_t \pmod{K[[t]] \widetilde{\eta}_t}, $$
 where $B_p(a, a)$ is the $p$-adic Beta function defined in Section \ref{definition}. 
\end{prop}

\begin{proof}
First, we prove that 
$$\Phi_{X/S, \sigma}(\widetilde{\eta}_t) \in K \widetilde{\eta}_t. $$ 
Since $\Phi_{X/S, \sigma} \nabla = \nabla \Phi_{X/S, \sigma}$, we have $\Phi_{X/S, \sigma} \operatorname{Ker} (\nabla) \subset \operatorname{Ker} (\nabla)$. 
By Proposition \ref{GM}, we have 
$$\operatorname{Ker}(\nabla)=K ( \widetilde{\omega}_t + g(t)\widetilde{\eta}_t) \oplus K \widetilde{\eta}_t, $$
where $g(t) \in tK[[t]]$ is the power series such that  
$$ g'(t)={(1-t^3)^{-1}F_{\frac13, \frac13; \frac23}(t^3)^{-2}}. $$
Hence, we have 
$$\Phi_{X/S, \sigma}(\widetilde{\eta}_t) = C_1( \widetilde{\omega}_t + g(t)\widetilde{\eta}_t)+C_2 \widetilde{\eta}_t$$
for some $C_1$, $C_2 \in K$. 
By considering the fiber at $t=0$, we have 
$$\Phi_{X/S, \sigma}(\widetilde{\eta}_t|_{t=0})=C_1 \widetilde{\omega}_t|_{t=0}+ C_2 \widetilde{\eta}_t|_{t=0}. $$
Note that $X_0$ is the Fermat cubic. 
By \cite[Proposition 1.4]{Coleman} and the assumption $p \equiv 1 \pmod{3}$, we have $\Phi_{X/S, \sigma}(\widetilde{\eta}_t|_{t=0}) \in K \widetilde{\eta}_t|_{t=0}$, hence we conclude that $C_1=0$, i.e., $\Phi_{X/S, \sigma}(\widetilde{\eta}_t) \in K \widetilde{\eta}_t$. 

Secondly, we prove the second assertion. 
Let 
\begin{align} \label{Phi_omega}
\Phi_{X/S, \sigma}(\widetilde{\omega}_t)=f_1(t) \widetilde{\omega}_t + f_2(t) \widetilde{\eta}_t, \quad f_1(t), f_2(t) \in K[[t]]. 
\end{align}
Applying $\nabla$ to \eqref{Phi_omega}, we have 
$$f'_1(t) \widetilde{\omega}_t \equiv 0 \pmod{K[[t]] \widetilde{\eta}_t}, $$
which implies that $f_1(t)$ is a constant. 
To determine the constant, we again consider the fiber at $t=0$.  
By \cite[Theorem 1.7]{Coleman}, we have 
$$\Phi_{X/S, \sigma} (\widetilde{\omega}_t |_{t=0})=pB_p\left(\frac13, \frac13 \right)^{-1} \cdot \widetilde{\omega}_t |_{t=0}, $$
which finishes the proof. 
\end{proof}

\begin{thm}[Unit root formula] \label{unit_root_formula}
Suppose $\sigma(t)=t^p$. 
Put 
$$e_t^{\rm unit}:=F_{\frac13, \frac13;\frac23}(t^3)^{-1} \widetilde{\eta}_t. $$
Then, we have 
\begin{align} \label{unit_root}
e_t^{\rm unit} \in H^1_{\rm dR}(X/S) \otimes_{\mathscr{O}(S)} K \langle t, (1-t^3)^{-1}, ([F_{\frac13, \frac13; \frac23}(z)]_{<p}|_{z=t^3})^{-1} \rangle
\end{align}
and 
\begin{align} \label{Frob_on_e}
\Phi_{X/S, \sigma}(e_t^{\rm unit})= B_p\left(\frac13, \frac13 \right) \mathscr{F}^{\rm Dw}_{\frac13, \frac13; \frac23}(t^3) e_t^{\rm unit}. 
\end{align}
\end{thm}

\begin{proof}
Since 
\begin{align*}
\dfrac{[F_{\frac13, \frac13; \frac23}(t^3)]'}{F_{\frac13, \frac13; \frac23}(t^3)} \in  K \langle t, (1-t^3)^{-1}, ([F_{\frac13, \frac13; \frac23}(z)]_{<p}|_{z=t^3})^{-1} \rangle
\end{align*}
 by Theorem \ref{theorem:1} (see the proof of \cite[Theorem 2.4]{As2}), \eqref{unit_root} follows. 
 We will show \eqref{Frob_on_e}, which is equivalent to 
 \begin{align*} 
 \Phi_{X/S, \sigma}(\widetilde{\eta}_t)=B_p\left(\frac13, \frac13 \right) \widetilde{\eta}_t. 
 \end{align*}
Let 
$$Q \colon H^1_{\rm dR}(X_K/S_K) \otimes H^1_{\rm dR}(X_K/S_K) \to H^2_{\rm dR}(X_K/S_K) \simeq \mathscr{O}(S_K)$$
be the cup-product pairing, which is anti-symmetric and non-degenerate. 
Since $\nabla(\widetilde{\omega}_t) \in K[[t]] \widetilde{\eta}_t$, we have 
$$dQ(\widetilde{\omega}_t, \widetilde{\eta}_t)=Q(\nabla(\widetilde{\omega}_t), \widetilde{\eta}_t)=0, $$
which means that $Q(\widetilde{\omega}_t, \widetilde{\eta}_t)=C$ for some $C \in K^{\times}$.  
Since the pairing $Q$ and the elements $\widetilde{\omega}_t$, $\widetilde{\eta}_t$ are defined over a ring $\Q[[t]]$, it turns out that $C \in K^{\times} \cap \Q[[t]]=\Q^{\times}$.  
 By Proposition \ref{Frobenius}, there is a constant $\alpha \in K$ such that $\Phi_{X/S, \sigma}(\widetilde{\eta}_t)=\al \widetilde{\eta}_t$. 
 Then we have 
 \begin{align*}
 Q(\Phi_{X/S, \sigma}(\widetilde{\omega}_t), \Phi_{X/S, \sigma}(\widetilde{\eta}_t))=pQ(\widetilde{\omega}_t, \widetilde{\eta}_t)^{\sigma}=pC, 
 \end{align*}
 hence 
 $$\al Q(\Phi_{X/S, \sigma}(\widetilde{\omega}_t), \widetilde{\eta}_t) =pC. $$
Since $Q(\widetilde{\eta}_t, \widetilde{\eta}_t)=0$, by Proposition \ref{Frobenius}, 
the left-hand side is 
$$p \al B_p\left(\frac13, \frac13 \right)^{-1} Q(\widetilde{\omega}_t, \widetilde{\eta}_t)= p \al B_p\left(\frac13, \frac13 \right)^{-1} C, $$
which means that $\al=B_p\left(1/3, 1/3 \right)$. Hence, the assertion follows. 
\end{proof}

\subsection{Syntomic regulator of the Hesse cubic curve}
For $\alpha \in W \cup \{\infty\}$, let $P_{\alpha}$ denote the $W$-valued point of $\mathbb{P}^1_{W}$ given by $t= \alpha$. 
Let $Q:=\mathbb{P}^1_{W} \setminus \{1, \zeta, \zeta^2\} \supset S=\mathbb{P}^1_{W} \setminus \{1, \zeta, \zeta^2, \infty\}$ and put $Y:= \of^{-1}(Q)$. 
Let $\overline{D} \subset \overline{Y}$ be the closure of the sections $\{x_0y_0z_0=0\}$. 
Each irreducible component of $\overline{D}$ is isomorphic to $\mathbb{P}^1_W$, and the nine curves are pairwise disjoint. 
Moreover, they meet the fiber $\of^{-1}(P_i)$ $(i \in \{1, \zeta, \zeta^2, \infty\})$ only in its regular locus, and the intersections are transverse.  
In particular, $\overline{D} \cup \bigcup_{i \in \{1, \zeta, \zeta^2, \infty\}} \of^{-1}(P_i)$ is a relative simple normal crossings divisor over $W$. 
Put $D:= Y \cap \overline{D}$, $D_X:= D \cap X$ and $U:=X \setminus D_X$. 
There is a $p$th Frobenius $\sigma$ on the weak completion $\mathscr{O}(S)^{\dagger}$ given by $\sigma(t)=ct^p$ compatible with the Frobenius on $W$. 
Then, the above setting $(Y/Q, S, D, \sigma)$ satisfies the conditions i), \ldots, iv) in \cite[Section 4.1]{AM} (note (iv) follows from \cite[Lemma 4.1]{AM}), hence 
 \begin{align*}
 [-]_{U/S} \colon K_2^{M}(\mathscr{O}(U)) \to \operatorname{Ext}^1_{\text{Fil-}F\text{-MIC}(S, \sigma)}(\mathscr{O}, H^1(U/S)(2))
 \end{align*}
is defined. 
Let 
$$\xi=6\{x, y\} \in K_2^M(\mathscr{O}(U))$$
be a Milnor symbol in $K_2$. 
Then, $\xi$ has no boundary at $X \setminus U$, i.e. the image of the tame symbol map is $1$,  hence the symbol map $[-]_{U/S}$ also defines a $1$-extension 
\begin{align*}
0 \to H^1(X/S)(2) \to M_{\xi}(X/S) \to \mathscr{O} \to 0
\end{align*}
in the exact category $\text{Fil-}F\text{-MIC}(S, \sigma)$ (\cite[Proposition 4.3]{AM}) endowed with 
\begin{itemize}
\item Frobenius $\Phi_{\sigma}$-action which is a $\sigma$-linear, 
\item $\operatorname{Fil}^i \subset M_{\xi}(X_K/S_K)_{\rm dR}$ (Hodge filtration) with 
$$ \operatorname{Fil}^0 M_{\xi}(X_K/S_K)_{\rm dR} \overset{\simeq}{\to} \mathscr{O}(S_K). $$
\end{itemize}
Let $e_{\xi} \in \text{Fil}^0 M_{\xi}(X_K/S_K)_{\rm dR}$ be the unique lifting of $1 \in \mathscr{O}(S_K)$.

Let $\we_k(t)$ and $\wE_k(t)$ ($k=1, 2$) be defined by 
\begin{align}
e_{\xi} - \Phi_{\sigma}(e_{\xi}) 
&=\we_1(t) \omega_t + \we_2(t) \eta_t \\ 
&=\wE_1(t) \widetilde{\omega}_t + \wE_2(t) \widetilde{\eta}_t. \label{E}
\end{align}
The relation between $\we_k(t)$ and $\wE_k(t)$ is explicitly given by 
\begin{align*}
&\we_1(t)=\wE_1(t) F_{\frac13, \frac13; \frac23}(t^3)^{-1} + (-t^2F_{\frac13, \frac13; \frac23}(t^3) +(1-t^3) [F_{\frac13, \frac13; \frac23}(t^3)]') \wE_2(t),  \\
&\we_2(t)=\wE_2(t) F_{\frac13, \frac13; \frac23}(t^3). 
\end{align*}
By the definition, $\we_k(t)$ $(k=1, 2)$ are automatically overconvergent functions, 
$$\we_k(t) \in K[t, (1-t^3)^{-1}]^{\dagger}. $$
On the other hand, since $[F_{\frac13, \frac13; \frac23}(t^3)]'/F_{\frac13, \frac13; \frac23}(t^3)$ is a convergent function, so is $\wE_1(t)/F_{\frac13, \frac13; \frac23}(t^3)$, 
$$\dfrac{\wE_1(t)}{F_{\frac13, \frac13; \frac23}(t^3)} \in K \langle t, (1-t^3)^{-1}, ([F_{\frac13,\frac13;\frac23}(z)]_{<p}|_{z=t^3})^{-1} \rangle. $$

The following is the main theorem in this paper, which provides a geometric aspect of $\widehat{\mathscr{F}}^{(\sigma)}_{a, a;2a}(t)$, the $p$-adic hypergeometric function defined in Section \ref{definition}.

\begin{thm} \label{main:2} 
Let $\sigma$ (resp. $\tau$) be the $p$th Frobenius defined by $\sigma(t)=ct^p$ (resp. $\tau(t)=c^3t^p$) with $c \in 1+pW$. Then, we have 
\begin{align*}
\dfrac{\wE_1(t)}{F_{\frac13,\frac13;\frac23} (t^3)} = - 2t \widehat{\mathscr{F}}^{(\tau)}_{\frac13, \frac13;\frac23}(t^3). 
\end{align*}
Hence, 
\begin{align*}
e_{\xi} -\Phi_{\sigma}(e_{\xi}) \equiv -2t \widehat{\mathscr{F}}^{(\tau)}_{\frac13, \frac13;\frac23}(t^3) \omega_t \pmod{K[[t]] \widetilde{\eta}_t}. 
\end{align*}
\end{thm}

\begin{proof}
Apply the Gauss-Manin connection $\nabla$ to \eqref{E}. 
Since $\nabla \Phi_{\sigma}= \Phi_{\sigma} \nabla$ and $\nabla(e_{\xi})=-d\log \xi=-6dt \wedge \omega_t$ by \cite[Remark 4.18]{N}, we have 
\begin{align*}
(1- \Phi_{\sigma}) (-6dt \wedge \omega_t)
=
\nabla(\wE_1(t) \widetilde{\omega}_t + \wE_2(t) \widetilde{\eta}_t). 
\end{align*}
Let $\Phi_{X/S, \sigma}$ denote the $p$-adic Frobenius on $H^1_{\rm rig} (X_{\overline{\mathbb{F}}_p}/S_{\overline{\mathbb{F}}_p})$. 
Then the $\Phi_{\sigma}$ on $H^1_{\rm rig}(X/S)(2)$ agrees with $p^{-2} \Phi_{X/S, \sigma}$ by the definition of the Tate twist. It follows from Proposition \ref{Frobenius} that we have 
\begin{align*}
\Phi_{X/S, \sigma} (\widetilde{\omega}_t) \equiv p B_p\left(\frac13, \frac13 \right)^{-1} \widetilde{\omega}_t \pmod{K[[t]] \widetilde{\eta}_t}. 
\end{align*}
Therefore, we have 
\begin{align*}
(1-\Phi_{\sigma}) (-6dt \wedge \omega_t)\equiv - 6\left(tF_{\frac13, \frac13; \frac23}(t^3) -B_p\left(\frac13, \frac13\right)^{-1} [tF_{\frac13, \frac13; \frac23}(t^3)]^{\sigma} \right) \cdot \dfrac{dt}{t} \widetilde{\omega}_t \pmod{K[[t]] \widetilde{\eta}_t}. 
\end{align*}
 On the other hand, by Proposition \ref{GM}, we have  
 \begin{align*}
\nabla(\wE_1(t) \widetilde{\omega}_t + \wE_2(t) \widetilde{\eta}_t) \equiv t \dfrac{d}{dt} \widehat{E}_1(t) \cdot \dfrac{dt}{t} \widetilde{\omega}_t  \pmod{K[[t]] \widetilde{\eta}_t}. 
 \end{align*}
 Hence, we obtain 
 \begin{align*}
 t\dfrac{d}{dt} \wE_1(t) = -6 \left(tF_{\frac13, \frac13; \frac23}(t^3) -B_p\left(\frac13, \frac13\right)^{-1} [tF_{\frac13, \frac13; \frac23}(t^3)]^{\sigma} \right). 
\end{align*}
Then we have 
\begin{align*}
-\wE_1(t) &= C+6 \int_0^t \left(tF_{\frac13, \frac13; \frac23}(t^3) -B_p\left(\frac13, \frac13\right)^{-1} [tF_{\frac13, \frac13; \frac23}(t^3)]^{\sigma}\right) \dfrac{dt}{t} \\
&=C+ 2\int_0^{t^3} \left(t^{\frac13}F_{\frac13, \frac13; \frac23}(t) -B_p\left(\frac13, \frac13\right)^{-1} [t^{\frac13}F_{\frac13, \frac13; \frac23}(t)]^{\tau} \right) \dfrac{dt}{t} \\
&\overset{(*)}{=}C +  2t\widehat{G}_{\frac13, \frac13; \frac23}^{(\tau)}(t^3), 
\end{align*}
where $(*)$ follows from $l=(p-1)/3 \equiv 0 \pmod{2}$. 
We will show that $C=0$. 
Note that $-\wE_1(t)/F_{\frac13, \frac13; \frac23}(t^3)$ and $2t \widehat{G}_{\frac13, \frac13; \frac23}^{(\tau)}(t^3)/F_{\frac13, \frac13; \frac23}(t^3)= 2t\widehat{\mathscr{F}}_{\frac13, \frac13;\frac23}^{(\tau)}(t^3)$ are convergent functions by \eqref{convergent}. 
Therefore, if $C \neq 0$, then we have 
$$\dfrac{1}{F_{\frac13, \frac13; \frac23}(t^3)} \in K\langle t, (1-t^3)^{-1}, ([F_{\frac13,\frac13;\frac23}(z)]_{<p}|_{z=t^3})^{-1} \rangle. $$
By the following lemma, this is a contradiction; thus $C$ must be $0$. 
Hence, the assertion follows. 
\end{proof}

\begin{lem} \label{HGconvergent}
Let $p$ be a prime and $N \ge 2$ an integer such that $p \equiv 1 \pmod{N}$. 
Let $a \in N^{-1}\Z$, $0 < 2a \le 1$. 
Then, $1/F_{a,a;2a}(t)$ is not a convergent function, i.e., 
$$\dfrac1{F_{a,a;2a}(t)} \not \in W \langle t, (1-t)^{-1}, ([F_{a,a;2a}(t)]_{<p})^{-1}\rangle. $$ 
\end{lem}

\begin{proof}
When $a=1/2$, this lemma is proved by Asakura (see the proof of \cite[Theorem 4.8]{As}). 
Suppose that $a \neq 1/2$, i.e. $0<2a<1$. 
For a power series $f(t)$, we write $\overline{f}(t) \equiv f(t) \pmod{p}$. 
By Theorem \ref{theorem:1}, we have 
\begin{align*} 
\dfrac{\overline{F}_{a, a;2a}(t)}{\overline{F}_{a, a;2a}(t^p)}  = [F_{a,a;2a}(t)]_{<p}
\end{align*}
in $A:=W/pW[t, (1-t)^{-1}, [F_{a,a;2a}(t)]_{<p}^{-1}]$.  
Put 
$$B:=\{\al \in \mathbb{F}_p \mid \al \neq 1 \ \text{and} \ [F_{a,a;2a}(t)]_{<p}|_{t=\al} \neq 0 \}. $$
Suppose that $1/F_{a, a;2a}(t)$ is a convergent function. 
Then we have $1/\overline{F}_{a, a;2a}(t) \in A$. 
For any $\al \in B \setminus \{0\}$, 
let $g(t)=1/\overline{F}_{a, a;2a}(t)=(t -\al)^k g_0(t)$ with $g_0(t) \in A$ and $g_0(\al) \neq 0$. 
Suppose $k \neq 0$. Then, we have  
$$ [F_{a,a;2a}(t)]_{<p}|_{t=\al} = \left. \dfrac{\overline{F}_{a, a;2a}(t)}{\overline{F}_{a, a;2a}(t^p)} \right|_{t=\al} = \left. \dfrac{(t^p-\al)^kg_0(t^p)}{(t-\al)^k g_0(t)} \right|_{t=\al} =(t-\al)^{k(p-1)}|_{t=\al}=0, $$
which is a contradiction. Hence, we have $k=0$. 
Note that $[F_{a,a;2a}(t)]_{<p}|_{t=0}=1$. 
Therefore, for any $\al \in B$, we have 
$$ [F_{a,a;2a}(t)]_{<p}|_{t=\al} = 1. $$
Let $l \in \{0, \ldots, p-1\}$ be the unique integer such that $a+l \equiv 0 \pmod{p}$. 
Since $|B| \ge p-l-1 >  l =\operatorname{deg}[F_{a,a;2a}(t)]_{<p}$, 
this is a contradiction.  
\end{proof}

By \cite[Theorem 4.4]{AM}, we have the following theorem.

\begin{thm} \label{syntomic_reg}
Let $\al \in W^{\times}$ satisfy $\al^3 \not \equiv 1 \pmod{p}$. 
Let $\sigma_{\al}$ (resp. $\tau_{\al}$) be the Frobenius defined by ${\sigma_{\al}}(t)=F(\alpha) \alpha^{-p} t^p$ (resp. ${\tau_{\al}}(t)=F(\alpha^3) \alpha^{-{3p}} t^p$), where $F$ is the Frobenius on $W$. 
Let 
$$r_{\rm syn} \colon K_2(X_{\alpha, K})^{(2)} \otimes \Q \to H^2_{\rm syn}(X_{\alpha, K}, \Q_p(2)) \simeq H^1_{\rm dR}(X_{\alpha, K}/K)$$
be the syntomic regulator map. Then 
$$r_{\rm syn}(\xi |_{X_{\alpha, K}})= \we_1(\al) \omega_{\alpha} + \we_2(\al) \eta_{\alpha}. $$
\end{thm}

\begin{cor} \label{main:cor}
Let the notations and assumptions be as in Theorem \ref{syntomic_reg}. 
Suppose that $[F_{\frac13, \frac13; \frac23}(t)]_{<p}|_{t=\al^3} \not \equiv 0 \pmod{p}$. 
Let $e_{\alpha}^{\rm unit}$ be as in Theorem \ref{unit_root_formula}.    
Let 
$$Q \colon H^1_{\rm dR}(X_{\alpha, K}/K) \otimes H^1_{\rm dR}(X_{\alpha, K}/K) \to H^2_{\rm dR}(X_{\alpha, K}/K) \simeq K$$
be the cup-product pairing. 
Then we have 
$$Q(r_{\rm syn}(\xi |_{X_{\alpha, K}}), e_{\alpha}^{\rm unit})=- 2\al \widehat{\mathscr{F}}_{\frac13, \frac13; \frac23}^{(\tau_{\alpha})}(\alpha^3) Q (\omega_{\alpha}, e_{\alpha}^{\rm unit}). $$
\end{cor}

\begin{proof}
Since $Q(e^{\rm unit}_{\alpha}, e^{\rm unit}_{\alpha})=0$, this follows from Theorems \ref{syntomic_reg} and \ref{main:2}. 
\end{proof}

\section{Transformation formula} \label{main:transform}
\subsection{Transformation formula}
Wang \cite[Conjecture 4.12]{Wang} conjectures the transformation formula between $\mathscr{F}^{(\sigma)}_{a, \ldots, a; \underline{1}}(t)$ and $\wF_{a, \ldots, a; \underline{1}}(t^{-1})$ as follows. 
Let $a \in \Z_p \setminus \Z_{\le 0}$ satisfy $a^{(r)}=a$ for some $r>0$.   
Define a polynomial by 
$$h_{a, \ldots, a; \underline{1}}(t)=\prod_{i=0}^{r-1} [F_{a^{(i)}, \ldots, a^{(i)};\underline{1}}(t)]_{<p}. $$
By \cite[Proposition 4.11(3)]{Wang}, we have 
$$W \langle t, t^{-1}, h_{a, \ldots, a; \underline{1}}(t^{-1})^{-1} \rangle = W \langle t, t^{-1}, h_{a, \ldots, a; \underline{1}}(t)^{-1} \rangle, $$ 
hence there is an involution 
$$\iota \colon W \langle t, t^{-1}, h_{a, \ldots, a; \underline{1}}(t^{-1})^{-1} \rangle \to W \langle t, t^{-1}, h_{a, \ldots, a; \underline{1}}(t)^{-1} \rangle; \quad f(t) \mapsto f(t^{-1}). $$

\begin{conj} \label{Wang_conj}
Let $\sigma(t)=ct^p$ and $\widehat{\sigma}(t)=c^{-1}t^p$ with $c \in 1+ pW$. 
Suppose that $a \in \Z_p \setminus \Z_{\le 0}$ satisfies $a^{(r)}=a$ for some $r>0$. 
Then, we have 
 \begin{align*}
\mathscr{F}^{(\sigma)}_{a, \ldots, a; \underline{1}}(t) = - \widehat{\mathscr{F}}_{a, \ldots, a; \underline{1}}^{(\widehat{\sigma})}(t^{-1})
 \end{align*}
 in the ring $W \langle t, t^{-1}, h_{a, \ldots, a; \underline{1}}(t)^{-1} \rangle$, where $\mathscr{F}^{(\sigma)}_{a, \ldots, a; \underline{1}}(t)$ is the $p$-adic hypergeometric function of logarithmic type defined in Section \ref{p_adic_logHGF} and 
 $\widehat{\mathscr{F}}_{a, \ldots, a; \underline{1}}^{(\widehat{\sigma})}(t^{-1})$ is defined as $\iota(\widehat{\mathscr{F}}_{a, \ldots, a; \underline{1}}^{(\widehat{\sigma})}(t))$.    
\end{conj}

It is known that Conjecture \ref{Wang_conj} holds modulo $p$ (\cite[Example 4.13]{Wang}). 
Wang (\cite[Theorem 4.14]{Wang}, \cite[Corollary 5.12]{Wang2}) proves that Conjecture \ref{Wang_conj} is true under the following two conditions: 
\begin{enumerate}
\item $a \in N^{-1} \Z$ and $p \nmid N$ for some $N \ge 2$; 
\item $0 < a < 1$. 
\end{enumerate}

In this paper, we give a conjecture of the transformation formula between our $p$-adic hypergeometric function $\wF_{a, a; 2a}(t)$ and a $p$-adic hypergeometric function of logarithmic type $\mathscr{F}^{(\sigma)}_{a, 1-a; 1}(t)$. 
Define polynomials by 
$$h_a(t)=\prod_{i=0}^{r-1}[F^{(i)}_{a,1-a;1}(t)]_{<p}, \quad \widehat{h}_a(t)=\prod_{i=0}^{r-1}[F^{(i)}_{a,a;2a}(t)]_{<p}. $$

\begin{prop} \label{involution}
There is an isomorphism 
$$\omega \colon W \langle t, t^{-1}, \widehat{h}_a(t)^{-1} \rangle \to W \langle t, t^{-1}, h_a(t)^{-1} \rangle; \quad \omega(f(t))=f(t^{-1}). $$
\end{prop}

\begin{proof}
It suffices to show that for any integer $n \geq 1$, there exists an isomorphism 
$$\omega_n \colon W/p^n[ t, t^{-1}, \widehat{h}_a(t)^{-1}] \to W/p^n[ t, t^{-1}, h_a(t)^{-1}] ; \quad f(t) \mapsto f(t^{-1}). $$
First, we prove that  
$$\widehat{h}_a(t^{-1}) \in (W/p^n[t, t^{-1}, h_a(t)^{-1}])^\times. $$
To prove this, it suffices to show that it is a unit modulo $pW[t, t^{-1}, h_a(t)^{-1}]$.
Let $l_i \in \{0, \ldots, p-1\}$ be the integer such that $a^{(i)}+l_i \equiv 0 \pmod{p}$. 
Then, by \cite[(15.8.6)]{NIST}, we have 
\begin{align} \label{Congruence_modp}
\begin{split}
[F^{(i)}_{a,a;2a}(z)]_{<p}|_{z=t^{-1}}   \equiv \dfrac{(a^{(i)})_{l_i}}{(2a^{(i)})_{l_i}} (-t)^{-l_i} [F^{(i)}_{a, 1-a; 1}(t)]_{<p} \pmod{p}. 
\end{split}
\end{align}
Here, we used $(2a)^{(i)}=2a^{(i)}$ and $(1-a)^{(i)}=1-a^{(i)}$. 
Since $(a^{(i)})_{l_i}/(2a^{(i)})_{l_i} \in W^{\times}$, we have 
\begin{align*}
([F^{(i)}_{a, a; 2a}(z)]_{<p}|_{z=t^{-1}})^{-1}=\dfrac{(2a^{(i)})_{l_i}}{(a^{(i)})_{l_i}} \dfrac{(-t)^{l_i}}{[F^{(i)}_{a,1-a;1}(t)]_{<p}} 
\end{align*}
in $W/pW[t, t^{-1}, h_a(t)^{-1}]$, i.e., $[F^{(i)}_{a, a; 2a}(z)]_{<p}|_{z=t^{-1}}$ is a unit in $W/pW[t, t^{-1}, h_a(t)^{-1}]$. 
Since $\widehat{h}_a(t^{-1})$ is the product of $[F^{(i)}_{a, a; 2a}(z)]_{<p}|_{z=t^{-1}}$, it is also a unit. 
Therefore, the map $\omega_n$ is defined.

Similarly, by replacing \(t\) with \(t^{-1}\) in \eqref{Congruence_modp}, we obtain a homomorphism in the opposite direction.
Thus, $\omega_n$ is an isomorphism for any $n \geq1$.
\end{proof}

\begin{conj} \label{main:conj}
We have 
\begin{align*}
\mathscr{F}^{(\sigma)}_{a, 1-a; 1}(t) = - \widehat{\mathscr{F}}^{(\widehat{\sigma})}_{a, a; 2a}(t^{-1})
\end{align*}
in the ring $W\langle t, t^{-1}, h_a(t)^{-1} \rangle$, where $\widehat{\mathscr{F}}^{(\widehat{\sigma})}_{a, a; 2a}(t^{-1})$ is defined by $\omega(\widehat{\mathscr{F}}^{(\widehat{\sigma})}_{a, a; 2a}(t))$ and $\mathscr{F}^{(\sigma)}_{a, 1-a; 1}(t)$ is the $p$-adic hypergeometric function of logarithmic type. 
\end{conj}

\begin{thm} \label{conj_mod_p}
Conjecture \ref{main:conj} holds modulo $p$. 
 \end{thm}

\begin{proof}
By \cite[Theorem 3.3]{As} and Theorem \ref{main:1}, we  have 
\begin{align*}
\mathscr{F}_{a, 1-a;1}^{(\sigma)}(t) \equiv \dfrac{G_{a,1-a;1}^{(\sigma)}(t)_{<p}}{F_{a,1-a;1}(t)_{<p}}, \quad \widehat{\mathscr{F}}_{a, a;2a}^{(\widehat{\sigma})}(t^{-1}) \equiv \left. \dfrac{\widehat{G}_{a,a;2a}^{(\widehat{\sigma})}(z)_{<p}}{F_{a,a;2a}(z)_{<p}} \right|_{z=t^{-1}} 
\end{align*}
modulo $p$. 
Let $F_1(t) \equiv F_{a,1-a;1}(t)_{<p}$, $G(t) \equiv G_{a, 1-a;1}^{(\sigma)}(t)_{<p}$, $F_2(t^{-1}) \equiv F_{a,a;2a}(z)_{<p}|_{z=t^{-1}}$ and $\widehat{G}(t^{-1}) \equiv \widehat{G}^{(\widehat{\sigma})}_{a, a; 2a}(z)_{<p}|_{z=t^{-1}}  \pmod{p}$.  
Then, by \eqref{Congruence_modp} and Lemma \ref{Beta_cong}, we have 
\begin{align*}
\dfrac{G(t)}{F_1(t)} + \dfrac{\widehat{G}(t^{-1})}{F_2(t^{-1})}  
&\equiv \dfrac{(-1)^l t^{-l}B_p(a,a)^{-1} G(t)+ \widehat{G}(t^{-1})}{(-1)^lB_p(a,a)^{-1} t^{-l}F_1(t)} \pmod{p}\\
&= \dfrac{t^{-l}B_p(a,a)^{-1} (G(t)+ (-1)^lt^{l}B_p(a,a) \widehat{G}(t^{-1}))}{B_p(a,a)^{-1} t^{-l}F_1(t)}. 
\end{align*}
We claim that $G(t)+ (-1)^l t^{l}B_p(a,a) \widehat{G}(t^{-1}) \equiv 0 \pmod{p}$. 
Note that 
$$G(t) = \sum_{k=0}^l B_k t^k, \qquad \widehat{G}(t^{-1}) = \sum_{k=0}^l \widehat{B}_k t^{-k}. $$
Hence, we have 
\begin{align*}
G(t)+ (-1)^lt^{l}B_p(a,a) \widehat{G}(t^{-1}) &= \sum_{k=0}^l B_k t^k + (-1)^l B_p(a,a) \sum_{k=0}^{l} \widehat{B}_{l-k} t^k \\
&= \sum_{k=0}^l  \left( B_k  + (-1)^l B_p(a,a) \widehat{B}_{l-k}  \right)t^k.  
\end{align*}
First, we consider the case $1 \le k \le l$. 
Note that 
$$B_k=\dfrac{1}{k} \cdot \dfrac{(a)_k(1-a)_k}{k!^2} = \dfrac{1}k\cdot \dfrac{(a)_k^2}{k!^2} \cdot \dfrac{(1-a)_k}{(a)_k}$$
and 
\begin{align*} 
(-1)^l B_p(a,a) \widehat{B}_{l-k}=\dfrac{(-1)^l B_p(a,a)}{a+l-k} \cdot \dfrac{(a)_{l-k}^2 }{(2a)_{l-k} (l-k)!}. 
\end{align*}
By Lemma \ref{Beta_cong}, we have 
$$B_p(a, a) \cdot \dfrac{l!}{(2a)_l} \equiv (-1)^l \pmod{p^2}, $$
thus we obtain  
$$(-1)^l B_p(a,a) \widehat{B}_{l-k} \equiv  -\dfrac{1}{k} \cdot \dfrac{(a)_{l-k}^2}{(l-k)!^2} \cdot \dfrac{(2a)_l}{(2a)_{l-k}} \cdot \dfrac{(l-k)!}{l!} \pmod{p}. $$
Since 
$$\dfrac{(2a)_l}{(2a)_{l-k}} = (2a+l-k) \cdots (2a+l-1) \equiv (-1)^k (1-a)_k \pmod{p}$$
and 
$$\dfrac{(l-k)!}{l!} = \dfrac1{l \cdots (l-k+1)} \equiv \dfrac{(-1)^k}{(a)_k} \pmod{p}, $$
we have 
\begin{align*}
(-1)^l B_p(a,a) \widehat{B}_{l-k} \equiv - \dfrac1k \cdot \dfrac{(a)_{l-k}^2(1-a)_k}{(l-k)!^2(a)_k} \pmod{p}. 
\end{align*}
Since we have 
$$\dfrac{(a)_k}{k!} \equiv (-1)^l \dfrac{(a)_{l-k}}{(l-k)!} \pmod{p}, $$
the first case follows.

Secondly, we consider the case $k=0$. 
We have 
\begin{align*}
&B_0 +(-1)^l B_p(a, a) \widehat{B}_l \\
&= \psi_p(a)+ \psi_p(1-a)+2 \gamma_p -p^{-1} \log(c) +(-1)^l\dfrac{B_p(a, a)}{a+l} \left(A_l -c^{-\frac{a+l}{p}} (-1)^lB_p(a, a)^{-1} \right). 
\end{align*}
Write $c=1+pz$ and $a+l=dp^n$ with $p \nmid d$. Then we have 
\begin{align*}
-p^{-1} \log(c) - \dfrac{c^{-\frac{a+l}{p}}}{a+l} = -p^{-1} (pz) -\dfrac{1}{dp^n} (1- dp^nz + \cdots)+pC' = -\dfrac1{a+l} +pC 
\end{align*}
for some $C, C' \in W$.  
Therefore, we have 
\begin{align*}
&B_0 +(-1)^l B_p(a, a) \widehat{B}_l \\
& = \psi_p(a)+ \psi_p(1-a)+2 \gamma_p + (-1)^l \dfrac{(B_p(a, a) A_l - (-1)^l)}{a+l} +pC.  
\end{align*}
By Lemma \ref{Beta_cong}, we have 
$$B_p(a, a) \dfrac{l!}{(2a)_l} \equiv (-1)^l \pmod{p^{2n}}, $$
and by the proof of \cite[Lemma 3.9]{Wang}, we have 
$$\dfrac{(a)_l^2}{l!^2} \equiv 1-2dp^n(\psi_p(1+l) + \gamma_p) \pmod{p^{2n}}. $$
Therefore, we have 
\begin{align*}
&B_0 +(-1)^l B_p(a, a) \widehat{B}_l \\
 &\equiv \psi_p(a)+ \psi_p(1-a)+2 \gamma_p - 2 (\psi_p(1+l) + \gamma_p) \pmod{p} \\ 
 &\overset{(*)}{\equiv} 2 \psi_p(a) -2 \psi_p(-l) \\
 &\overset{(**)}{\equiv}  0 \pmod{p}. 
\end{align*}
Here, $(*)$ follows from \cite[Theorem 2.6 (2)]{As}, and $(**)$ follows from $a \equiv -l \pmod p$ and \cite[(2.13)]{As}. 
Thus the second case follows. 
\end{proof}

\begin{thm} \label{main:3}
Let $p$ be a prime such that $p \equiv 1 \pmod{3}$. 
Let $\sigma(t)=ct^p$ and $\widehat{\sigma}(t)=c^{-1}t^p$ with $c \in 1+pW$.  
Then, we have 
\begin{align*}
\mathscr{F}^{(\sigma)}_{\frac13, \frac23; 1}(t) = - \widehat{\mathscr{F}}^{(\widehat{\sigma})}_{\frac13, \frac13; \frac23}(t^{-1})
\end{align*}
in the ring $W\langle t, t^{-1}, [F_{\frac13, \frac23; 1}(t)]_{<p}^{-1} \rangle$, where $\mathscr{F}^{(\sigma)}_{\frac13, \frac23; 1}(t)$ is the $p$-adic hypergeometric function of logarithmic type. 
\end{thm}

\subsection{Proof of Theorem \ref{main:3}}
Let $g_{K} \colon X_{K} \to \mathbb{P}^1_{K}$ be the Hesse family of cubic curves over $K$, whose fiber is given as $X_s=g_{K}^{-1}(s):=f_{K}^{-1}(s^{-1})$. 
Then, $X_s$ is written as  
$$s(x^3+y^3+1)-3xy=0. $$

\begin{dfn}
We define differential forms on $X_{K}$ by 
\begin{align*}
&\omega'_s = t \omega_t = \dfrac{dy}{sx^2-y} = \dfrac{-dx}{sy^2-x}, \\
&\eta'_s=xy \omega'_s = t\eta_t = \dfrac{xydy}{sx^2-y} = -\dfrac{xydx}{sy^2-x}. 
\end{align*}
\end{dfn}
Since $\omega'_s$ (resp. $\eta'_s$) is of the first (resp. second) kind, they define cohomology classes, and we denote them by the same letter.

Put 
\begin{align*}
&\widetilde{\omega}'_s := \dfrac1{F_{\frac13, \frac23; 1}(s^3)} \omega'_s, \\
&\widetilde{\eta}'_s := \left(s^3F_{\frac13, \frac23; 1}(s^3)-s(1-s^3) [F_{\frac13, \frac23; 1}(s^3)]' \right) \omega'_s -s^2 F_{\frac13, \frac23; 1}(s^3) \eta'_s,  
\end{align*}
where $[F_{\frac13, \frac23; 1}(s^3)]'$ denotes $\frac{d}{ds}F_{\frac13, \frac23; 1}(s^3)$.

\begin{prop}
Let $\nabla \colon H^1_{\rm dR} (X_{K}/S_{K}) \to \mathscr{O}(S_{K}) ds \otimes H^1_{\rm dR} (X_{K}/S_{K}) $ be the Gauss-Manin connection. 
It naturally extends on ${K}((s)) \otimes_{\mathscr{O}(S_{K})} H^1_{\rm dR} (X_{K}/S_{K})$ which we also write by $\nabla$. Then 
\begin{align}
&\nabla
\begin{pmatrix}
\omega'_s& \eta'_s
\end{pmatrix}
=\dfrac{ds}{(1-s^3)s}
\otimes 
\begin{pmatrix}
\omega'_s & \eta'_s
\end{pmatrix}
\begin{pmatrix}
s^3 & s \\
-s^2 & -(2-s^3)
\end{pmatrix}, \label{GM1'} \\
&\nabla
\begin{pmatrix}
\widetilde{\omega}'_s& \widetilde{\eta}'_s
\end{pmatrix}
=ds
\otimes 
\begin{pmatrix}
\widetilde{\omega}'_s & \widetilde{\eta}'_s
\end{pmatrix}
\begin{pmatrix}
0 & 0 \\
(1-s^3)^{-1}s^{-1}F_{\frac13, \frac23; 1}(s^3)^{-2} & 0
\end{pmatrix}. \label{GM2'}
\end{align}
\end{prop}

\begin{proof}
The first equation \eqref{GM1'} follows from \eqref{GM1} and the second equation \eqref{GM2'} follows from \eqref{GM1'} and the Picard-Fuchs equation 
$$\left(s(1-s^3)\dfrac{d^2}{ds^2} + (1-4s^3) \dfrac{d}{ds}- 2s^2\right)F_{\frac13, \frac23; 1}(s^3)=0. $$
\end{proof}

\begin{prop}
 Suppose that $\sigma$ is given by $\sigma(s)=cs^p$ with $c \in 1+ pW$. 
 Then we have 
 $$\Phi_{X/S, \sigma} (\widetilde{\eta}'_s) \in K \widetilde{\eta}'_s, \qquad \Phi_{X/S, \sigma} (\widetilde{\omega}'_s) \equiv p \widetilde{\omega}'_s \pmod{{K}((s)) \widetilde{\eta}'_s}. $$
\end{prop}

\begin{proof}
As in the proof of Proposition \ref{Frobenius}, one can show that there is a constant $C \in K$ such that 
$$\Phi_{X/S, \sigma} (\widetilde{\eta}'_s) \in K \widetilde{\eta}'_s, \qquad \Phi_{X/S, \sigma} (\widetilde{\omega}'_s) \equiv p C\widetilde{\omega}'_s \pmod{{K}((s)) \widetilde{\eta}'_s}. $$
By the similar discussions in the proof of \cite[Proposition 4.6]{As}, one can show that $C=1$, which finishes the proof. 
\end{proof}

\begin{rmk}
By the similar discussions in the proof of Theorem \ref{unit_root_formula}, we obtain $\Phi_{X/S, \sigma} (\widetilde{\eta}'_s)=\widetilde{\eta}'_s$. 
\end{rmk}

Let $\varepsilon_k(s)$ and $E_k(s)$ ($k=1, 2$) be defined by 
\begin{align*}
e_{\xi} - \Phi_{\sigma}(e_{\xi}) 
&=\varepsilon_1(s) \omega'_s + \varepsilon_2(s) \eta'_s \\ 
&=E_1(s) \widetilde{\omega}'_s + E_2(s) \widetilde{\eta}'_s. 
\end{align*}
The relation between $\e_k(s)$ and $E_k(s)$ is explicitly given by 
\begin{align*}
&\e_1(s)=E_1(s) F_{\frac13, \frac23; 1}(s^3)^{-1} + (s^3F_{\frac13, \frac23; 1}(s^3) - s(1-s^3) [F_{\frac13, \frac23; 1}(s^3)]') E_2(s),  \\
&\e_2(s)=-s^2 E_2(s)  F_{\frac13, \frac23; 1}(s^3). 
\end{align*}
By the definition, $\e_k(s)$ $(k=1, 2)$ are automatically overconvergent functions, i.e., 
$$\e_k(s) \in K[s, s^{-1}, (1-s^3)^{-1}]^{\dagger}. $$
On the other hand, since $[F_{\frac13, \frac23; 1}(s^3)]'/F_{\frac13, \frac23; 1}(s^3)$ is a convergent function (cf. \cite[Lemma 3.4]{Dw}), so is $E_1(s)/F_{\frac13, \frac23; 1}(s^3)$, i.e., 
$$\dfrac{E_1(s)}{F_{\frac13, \frac23; 1}(s^3)} \in K \langle s, s^{-1}, (1-s^3)^{-1}, ([F_{\frac13,\frac23; 1}(z)]_{<p}|_{z=s^3})^{-1} \rangle. $$

\begin{thm} \label{main:2'}
Let $\sigma$ (resp. $\tau$) be the $p$th Frobenius defined by $\sigma(s)=cs^p$ (resp. $\tau(s)=c^3s^p$) with $c \in 1+pW$. Then, we have 
\begin{align*}
\dfrac{E_1(s)}{F_{\frac13,\frac23;1} (s^3)} = 2\mathscr{F}^{(\tau)}_{\frac13, \frac23;1 }(s^3),  
\end{align*}
where $\mathscr{F}^{(\tau)}_{\frac13, \frac23;1 }(s^3)$ is the $p$-adic hypergeometric function of logarithmic type defined in Section \ref{p_adic_logHGF}. 
Hence, we have 
\begin{align*}
e_{\xi} -\Phi_{\sigma}(e_{\xi}) \equiv  2\mathscr{F}^{(\tau)}_{\frac13, \frac23;1}(s^3) \omega'_s \pmod{K((s))  \widetilde{\eta}'_s}. 
\end{align*}
\end{thm}

\begin{proof}
We compute 
$$\nabla(e_{\xi})= -d \log \xi= 6\dfrac{ds}{s} \wedge \omega'_s. $$
The rest is the same as in the proof of Theorem \ref{main:2} (see also the proof of \cite[Theorem 4.8]{As}).   
\end{proof}

\begin{proof}[Proof of Theorem \ref{main:3}]
Let $e_t^{\rm unit}$ be as in Theorem \ref{unit_root_formula}. 
Let $S'=\mathbb{P}_W^1 \setminus \{0, 1, \zeta, \zeta^2\}$. 
We can define 
$$e_s^{\rm unit}=F_{\frac13, \frac23;1}(s^3)^{-1}\widetilde{\eta}'_s \in H^1_{\rm dR}(X/S') \otimes_{\mathscr{O}(S')} K \langle s, s^{-1}, (1-s^3)^{-1}, ([F_{\frac13,\frac23; 1}(z)]_{<p}|_{z=s^3})^{-1} \rangle$$
similarly. 
Then, the isomorphism $\omega$ defined in Proposition \ref{involution} induces an isomorphism 
$$\omega \colon K  \langle t, t^{-1}, (1-t^3)^{-1}, [F_{\frac13,\frac13; \frac23}(z)]_{<p}|_{z=t^3})^{-1} \rangle \cdot e_t^{\rm unit} \xrightarrow{\simeq}  K \langle s, s^{-1}, (1-s^3)^{-1}, ([F_{\frac13,\frac23; 1}(z)]_{<p}|_{z=s^3})^{-1} \rangle \cdot e_s^{\rm unit}. $$

Let $\tau$ (resp. $\widehat{\tau}$) be the $p$th Frobenius defined by $\tau(t)=c^3t^p$ (resp. $\widehat{\tau}(t)=c^{-3} t^p$).  
By Theorems \ref{main:2} and \ref{main:2'}, we have 
$$ 2 \mathscr{F}^{(\tau)}_{\frac13, \frac23;1 }(s^3) \omega'_s \equiv \left.  - 2t \widehat{\mathscr{F}}^{(\widehat{\tau})}_{\frac13, \frac13;\frac23}(t^3) \omega_t \right|_{t=s^{-1}}$$
modulo $K \langle s, s^{-1}, (1-s^3)^{-1}, ([F_{\frac13,\frac23; 1}(z)]_{<p}|_{z=s^3})^{-1} \rangle \cdot e_s^{\rm unit}$. 
Since we have 
$$\left. t\omega_t \right|_{t=s^{-1}}= \dfrac1s \cdot \dfrac{dy}{x^2-y/s}= \omega'_s, $$
we obtain 
$$\mathscr{F}^{(\tau)}_{\frac13, \frac23;1 }(s^3) = - \widehat{\mathscr{F}}^{(\widehat{\tau})}_{\frac13, \frac13;\frac23}(s^{-3}).  $$
Since $p \neq 3$, the map 
$$\varphi \colon 1+pW \to 1+pW; \quad c \mapsto c^3$$
is an isomorphism. Hence, for any $C\in1+pW$, there exists $c\in1+pW$ such that $C=c^3$. Therefore, the theorem follows. 
\end{proof}

\section*{Acknowledgment}
The author used ChatGPT (GPT-5.6 Sol, OpenAI) to improve the English and to check the proofs for errors. The author takes full responsibility for the mathematical correctness and the final content of the manuscript.
The author is supported by JSPS KAKENHI Grant (JP26K16957) and Waseda University Grant for
Special Research Projects (Project number: 2026C-262).


\begin{thebibliography}{99}
\bibitem{As4}
M. Asakura, 
{\it Regulators of  $K_2$  of hypergeometric fibrations}, 
Res. Number Theory \textbf{4}, 2018, no. 2, Paper No. 22, 25 pp.


 \bibitem{As2}
 M. Asakura, {\it Frobenius action on a hypergeometric curve and an algorithm for computing values of Dwork's  $p$-adic hypergeometric functions}, 
Springer Proc. Math. Stat., \textbf{373}
Springer, Cham, 2021, 1--45.


 \bibitem{As} 
M. Asakura, {\it New $p$-adic hypergeometric functions and syntomic regulators}, 
J. Th\'eor. Nombres Bordeaux \textbf{35}, 2023, no. 2, 393--451.
 

\bibitem{As3}
M. Asakura, {\it A generalization of the Ross symbols in higher K-groups and hypergeometric functions II}, 
arXiv:2102.07946v3


\bibitem{AM}
M. Asakura, K. Miyatani, 
{\it Milnor  $K$-theory,  F-isocrystals and syntomic regulators}, 
J. Inst. Math. Jussieu \textbf{23}, 2024, no. 3, 1357--1415.


\bibitem{Besser}
A. Besser, 
{\it Syntomic regulators and $p$-adic integration. I. Rigid syntomic regulators,} 
Israel J. Math. {\bf 120} (2000), 291--334. 


\bibitem{Coleman}
R.~F. Coleman, 
{\it On the Frobenius matrices of Fermat curves}, 
$p$-adic analysis (Trento, 1989), 173--193, Lecture Notes in Math., \textbf{1454}, Springer, Berlin. 

 
\bibitem{Deninger}
C. Deninger, 
\textit{Deligne periods of mixed motives, $K$-theory and the entropy of certain $\Z_n$-actions},
J. Amer. Math. Soc. \textbf{10} no. 2, 259--281 (1997)
 
 
\bibitem{Dw}
B. Dwork, {\it  $p$-adic cycles}, Inst. Hautes \'Etudes Sci. Publ. Math. No. 37 (1969), 27--115.
 
\bibitem{Dw2} 
B. Dwork, {\it On $p$-adic differential equations IV}, 
Ann. Sci. Ecole Norm. Sup. tome \textbf{6}, no 3, 1973, p. 295--316.


\bibitem{Koblitz}
N. Koblitz, 
{\it $p$-adic numbers, $p$-adic analysis, and zeta-functions}, 
Grad. Texts in Math., \textbf{58}, 
Springer-Verlag, New York, 1984, xii+150 pp.


\bibitem{Morita}
Y. Morita, 
{\it A $p$-adic analogue of the $\Gamma$-functions}, 
J. Fac. Sci Univ. Tokyo \textbf{22} (1975), 256--266.


\bibitem{Nekovar_Niziol}
J. Nekov\'a\v r{} and W. Nizio\l, 
{\it Syntomic cohomology and $p$-adic regulators for varieties over $p$-adic fields}, 
Algebra Number Theory {\bf 10} (2016), no.~8, 1695--1790. 


\bibitem{N2}
Y. Nemoto, 
{\it On special values of generalized  $p$-adic hypergeometric functions of logarithmic type}, 
Ramanujan J. \textbf{70} (2026), no. 2, Paper No. 31, 27 pp. 


\bibitem{N3}
Y. Nemoto, 
{\it  Transformation formula of Dwork’s $p$-adic hypergeometric function}, 
 Bulletin of the Australian Mathematical Society,\textbf{113} (2026), no.3, 553--564.


\bibitem{N}
Y. Nemoto, 
{\it Regulator of the Hesse cubic curves and hypergeometric functions}, 
Manuscripta Math. \textbf{175} (2024), no. 3--4, 813--840.


\bibitem{Rod} 
F. Rodriguez Villegas, 
\textit{Modular Mahler measures. I},
Topics in number theory (University Park, PA, 1997), Math. Appl., \textbf{467}, Kluwer Acad. Publ., Dordrecht, 17--48 (1999)


\bibitem{Wang}
C. H. Wang, 
\textit{Congruence relations for  $p$-adic hypergeometric functions  $\widehat{\mathscr{F}}^{(\sigma)}_{a, \ldots, a}(t)$  and its transformation formula}, 
Manuscripta Math. \textbf{169} (2022), no. 3--4, 565--602.


\bibitem{Wang2} 
C. H. Wang, 
\textit{On transformation formulas of $p$-adic hypergeometric functions}, arXiv:2104.14092v3. 


\bibitem{Young}
P. T. Young, {\it Ap\'ery numbers, Jacobi sums, and special values of generalized  $p$-adic hypergeometric functions}, 
J. Number Theory \textbf{41}, 1992, no. 2, 231--255.


\bibitem{NIST}
{\it NIST Handbook of Mathematical Functions}, Edited by Frank W. J. Olver, Daniel W. Lozier, Ronald
F. Boisvert and Charles W. Clark. Cambridge Univ. Press, 2010.

\end{thebibliography}
\end{document}